\documentclass{amsart}

\usepackage{graphicx,verbatim, amsmath, amssymb, amsthm, amsfonts,hyperref,amsxtra}
\usepackage{ifthen}

\newtheorem{proposition}{Proposition}[section]
\newtheorem{theorem}[proposition]{Theorem}
\newtheorem{corollary}[proposition]{Corollary}
\newtheorem{lemma}[proposition]{Lemma}

\theoremstyle{definition}
\newtheorem{example}[proposition]{Example}

\theoremstyle{remark}
\newtheorem{remark}[proposition]{Remark}

\numberwithin{equation}{section}

\newcommand{\newword}[1]{\textbf{\textit{#1}}}

\newcommand{\integers}{\mathbb Z}
\newcommand{\reals}{\mathbb R}

\newcommand{\ep}{\epsilon}
\newcommand{\set}[1]{{\lbrace #1 \rbrace}}
\newcommand{\Set}[1]{{\big\lbrace #1 \big\rbrace}}
\newcommand{\SEt}[1]{{\Big\lbrace #1 \Big\rbrace}}

\newcommand{\join}{\vee}
\renewcommand{\Join}{\bigvee}
\newcommand{\meet}{\wedge}

\newcommand{\inv}{\operatorname{inv}}
\renewcommand{\neg}{\operatorname{neg}}
\newcommand{\rev}{\operatorname{rev}}
\newcommand{\alt}{\operatorname{alt}}
\newcommand{\Irr}{\operatorname{Irr}}
\newcommand{\Con}{\operatorname{Con}}
\newcommand{\Camb}{\operatorname{Camb}}
\newcommand{\Cg}{\operatorname{Cg}}
\newcommand{\covered}{\lessdot}
\newcommand{\A}{\mathcal{A}}
\newcommand{\F}{\mathcal{F}}

\newcommand{\ck}{\spcheck}
\newcommand{\tfg}{\tilde{\mathfrak{g}}}
\newcommand{\fg}{\mathfrak{g}}
\newcommand{\fh}{\mathfrak{h}}
\newcommand{\fr}{\mathfrak{r}}
\newcommand{\fn}{\mathfrak{n}}
\newcommand{\tfn}{\mathfrak{n}}
\newcommand{\ad}{\operatorname{ad}}

\author{Nathan Reading}
\address{Department of Mathematics, North Carolina State University, Raleigh, NC, USA}

\subjclass[2010]{20F55, 06B10}

\thanks{The author was partially supported by NSA grant H98230-09-1-0056 and by NSF grant DMS-1101568.}

\title{Lattice homomorphisms between weak orders}

\begin{document}

\begin{abstract}
We classify surjective lattice homomorphisms $W\to W'$ between the weak orders on finite Coxeter groups.
Equivalently, we classify lattice congruences $\Theta$ on $W$ such that the quotient $W/\Theta$ is isomorphic to $W'$.
Surprisingly, surjective homomorphisms exist quite generally:
They exist if and only if the diagram of $W'$ is obtained from the diagram of $W$ by deleting vertices, deleting edges, and/or decreasing edge labels.
A surjective homomorphism $W\to W'$ is determined by its restrictions to rank-two standard parabolic subgroups of~$W$.
Despite seeming natural in the setting of Coxeter groups, this determination in rank two is nontrivial.
Indeed, from the combinatorial lattice theory point of view, all of these classification results should appear unlikely \textit{a priori}.
As an application of the classification of surjective homomorphisms between weak orders, we also obtain a classification of surjective homomorphisms between Cambrian lattices and a general construction of refinement relations between Cambrian fans.
\end{abstract}

\maketitle
\setcounter{tocdepth}{1}
\tableofcontents

\section{Introduction}\label{intro}
The weak order on a finite Coxeter group $W$ is a partial order (in fact, lattice~\cite{orderings}) structure on $W$ that encodes both the geometric structure of the reflection representation of $W$ and the combinatorial group theory of the defining presentation of~$W$.
Recent papers have elucidated the structure of lattice congruences on the weak order \cite{congruence} and applied this understanding to construct fans coarsening the normal fan of the $W$-permutohedron~\cite{con_app}, 
combinatorial models of cluster algebras of finite type \cite{cambrian,sort_camb,camb_fan}, polytopal realizations of generalized associahedra \cite{HL,HLT}, and sub Hopf algebras of the Malvenuto-Reutenauer Hopf algebra of permutations \cite{sash,rectangle,Meehan,con_app}.  
A~thorough discussion of lattice congruences of the weak order (and more generally of certain posets of regions) is available in \cite{regions9,regions10}.

The purpose of this paper is to classify surjective lattice homomorphisms between the weak orders on two finite Coxeter groups $W$ and $W'$.
Equivalently, we classify the lattice congruences $\Theta$ on a finite Coxeter group $W$ such that the quotient lattice $W/\Theta$ is isomorphic to the weak order on a finite Coxeter group~$W'$.

From the point of view of combinatorial lattice theory, the classification results are quite surprising \textit{a priori}.
As an illustration of the almost miraculous nature of the situation, we begin this introduction with a representative example (Example~\ref{miraculous}), after giving just enough lattice-theoretic details to make the example understandable.
(More lattice-theoretic details are in Section~\ref{shard sec}.)

A \newword{homomorphism} from a lattice $L$ to a lattice $L'$ is a map $\eta:L\to L'$ such that $\eta(x\meet y)=\eta(x)\meet\eta(y)$ and $\eta(x\join y)=\eta(x)\join\eta(y)$.
A \newword{congruence} on a lattice $L$ is an equivalence relation such that 
\[(x_1\equiv y_1\text{ and }x_2\equiv y_2)\text{ implies }\left[(x_1\meet x_2)\equiv (y_1\meet y_2)\text{ and }(x_1\join x_2)\equiv (y_1\join y_2)\right].\]
Given a congruence $\Theta$ on $L$, the \newword{quotient lattice} $L/\Theta$ is a lattice structure on the set of equivalence classes where the meet $C_1\meet C_2$ of two classes is the equivalence class containing $x\meet y$ for any $x\in C_1$ and $y\in C_2$ and the join is described similarly.
When $L$ is a finite lattice, the congruence classes of any congruence $\Theta$ on $L$ are intervals.
The quotient $L/\Theta$ is isomorphic to the subposet of $L$ induced by the set of elements $x$ such that $x$ is the bottom element of its congruence class.

We use the symbol $\covered$ for cover relations in $L$ and often call a pair $x\covered y$ an \newword{edge} (because it forms an edge in the Hasse diagram of $L$).
If $x\covered y$ and $x\equiv y$, then we say that the congruence \newword{contracts} the edge $x\covered y$.
Since congruence classes on a finite lattice are intervals, to specify a congruence it is enough to specify which edges the congruence contracts.
Edges cannot be contracted independently;  rather, contracting some edge typically forces the contraction of other edges to ensure that the result is a congruence.
Forcing among edge contractions on the weak order is governed entirely\footnote{In a general lattice, forcing might be less local. (See \cite{GratzerPolygon}.)  
The weak order is special because it is a \newword{polygonal lattice}.
See \cite[~Definition~9-6.1]{regions9}, \cite[~Theorem~9-6.5]{regions9}, and \cite[~Theorem~10-3.7]{regions10}.} by a local forcing rule in \newword{polygons}.
A polygon is an interval such that the underlying graph of the Hasse diagram of the interval is a cycle.
There are two \newword{top edges} in a polygon, the two that are incident to the maximum, and two \newword{bottom edges}, incident to the minimum.
The remaining edges in the interval, if there are any, are \newword{side edges}.
The forcing rule for polygons is the following:  if a top (respectively bottom) edge is contracted, then the opposite bottom (respectively top) edge must also be contracted, and all side edges (if there are any) must be contracted.
One case of the rule is illustrated in Figure~\ref{poly force}, where shading indicates contracted edges.
(The other case of the rule is dual to the illustration.)
\begin{figure}
\scalebox{.9}{\includegraphics{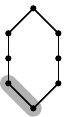}}\raisebox{21 pt}{\,\,\,\,$\implies$\,\,\,\,}\scalebox{.9}{\includegraphics{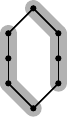}}
\caption{The forcing rule for edge contractions in the weak order}
\label{poly force}
\end{figure}

\begin{example}\label{miraculous}
Consider a Coxeter group $W$ of type $B_3$.
Figure~\ref{weakB3diagram}.a is a close-up of a certain order ideal in the weak order on $W$.
\begin{figure}
\begin{tabular}{ccc}
\scalebox{.85}{\includegraphics{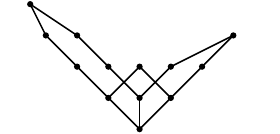}}&&\scalebox{.85}{\includegraphics{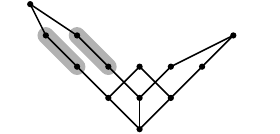}}\\
(a)&&(b)
\end{tabular}
\caption{a:  The defining presentation of $W$, encoded in an order ideal in the weak order.  b:  Contracting two edges in the order ideal.}
\label{weakB3diagram}
\end{figure}
This ideal contains all of the information about the Coxeter diagram of $W$.
Namely, the presence of an octagon indicates an edge with label 4, the hexagon indicates an edge with label 3, and the square indicates a pair of vertices not connected by an edge.

The Coxeter diagram of a Coxeter group of type $A_3$ has the same diagram, except that the label 4 is replaced by 3.
Informally, we can turn the picture in Figure~\ref{weakB3diagram}.a into the analogous picture for $A_3$, by contracting two side edges of the octagon to form a hexagon, as indicated by shading in Figure~\ref{weakB3diagram}.b.
If we take the same two edges in the whole weak order on $W$, we can use the polygonal forcing rules to find the finest lattice congruence that contracts the two edges.
This congruence is illustrated in Figure~\ref{B3toA3}.a.
\begin{figure}
\begin{tabular}{ccc}
\scalebox{.85}{\includegraphics{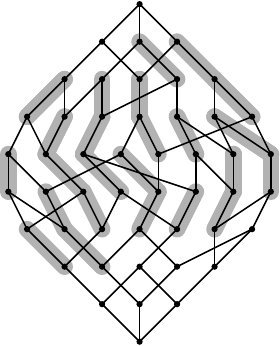}}&&\scalebox{.85}{\includegraphics{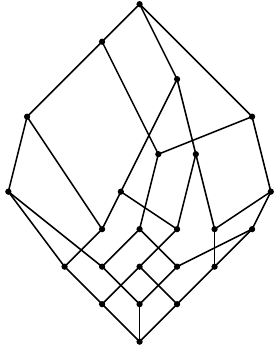}}\\
(a)&&(b)
\end{tabular}
\caption{a:  The smallest congruence on $W$ contracting the edges shaded in Figure~\ref{weakB3diagram}.b.  b:  The quotient modulo this congruence.}
\label{B3toA3}
\end{figure}
\textit{A~priori}, we shouldn't expect this congruence to have any significance, but, surprisingly, the quotient, shown in Figure~\ref{B3toA3}.b, of the weak order modulo this congruence is isomorphic to the weak order on a Coxeter group of type $A_3$.
(Recall from above that the lattice quotient is isomorphic to the subposet consisting of elements that are at the bottom of their congruence class.)

To recap:  
We start with the weak order on $B_3$, look at some polygons at the bottom of the weak order that encode the Coxeter diagram for $B_3$, and na\"{i}vely contract edges of these polygons to make one of the polygons smaller so that the polygons instead encode the Coxeter diagram for $A_3$.
Miraculously, the contracted edges generate a congruence such that the quotient is the weak order on $A_3$.
\end{example}

In general, a \newword{diagram homomorphism} starts with the Coxeter diagram of a Coxeter system $(W,S)$, deletes vertices, decreases labels on edges, and/or erases edges, and relabels the vertices to obtain the Coxeter diagram of some Coxeter system $(W',S')$.
(When no vertices are deleted, no labels are decreased, and no edges are erased, this is a \newword{diagram isomorphism} and when $(W',S')=(W,S)$ it is a \newword{diagram automorphism}.)
For brevity in what follows, we will say ``a diagram homomorphism from $(W,S)$ to $(W',S')$'' to mean ``a diagram homomorphism from the Coxeter diagram of $(W,S)$ to the Coxeter diagram of $(W',S')$.''

The first main results of the paper are the following theorem and several more detailed versions of it.
\begin{theorem}\label{main}
Given finite Coxeter systems $(W,S)$ and $(W',S')$, there exists a surjective lattice homomorphism from the weak order on $W$ to the weak order on $W'$ if and only if there exists a diagram homomorphism from $(W,S)$ to $(W',S')$
\end{theorem}

\begin{remark}\label{auto}
A restriction of Theorem~\ref{main} to \emph{isomorphisms} is well-known and extends to a characterization of meet-semilattice isomorphisms of the weak order on many infinite Coxeter groups.
See \cite[Corollary~3.2.6]{Bj-Br}.
\end{remark}

\begin{remark}\label{what's hard}
The existence of surjective homomorphisms between weak orders is not difficult to prove, \textit{a posteriori}:
We give explicit maps.
However, the other results in the classification need more machinery, specifically the machinery of shards, as explained in Section~\ref{shard sec}.
Furthermore, without the machinery of shards, we would not be able to find the explicit homomorphisms that prove existence.
\end{remark}

In order to make more detailed classification statements, we first give a factorization result for any surjective lattice homomorphism between weak orders.
Given a finite Coxeter group $W$ and a standard parabolic subgroup $W_J$, the \newword{parabolic homomorphism} $\eta_J$ is the map taking $w\in W$ to $w_J\in W_J$, where $w_J$ is the parabolic factor in the usual factorization of $w$ as an element of the parabolic subgroup times an element of the quotient.
A parabolic homomorphism corresponds to a diagram homomorphism that only deletes vertices from the diagram of~$W$.

An \newword{atom} in a finite lattice is an element that covers the minimal element $\hat0$.
We will call a homomorphism of finite lattices \newword{compressive} if it is surjective and restricts to a bijection between the sets of atoms of the two lattices.
(The term is an analogy to the physical process of compression where atoms are not created or destroyed but are brought closer together.
If $\eta:L\to L'$ is compressive, then $\eta$ moves two atoms $a_1,a_2$ of $L$ weakly closer in the sense that the interval below $\eta(a_1)\join\eta(a_2)$ has weakly fewer elements than the interval below $a_1\join a_2$.)
In particular, a compressive homomorphism between weak orders on Coxeter systems $(W,S)$ and $(W',S')$ is a surjective homomorphism $W\to W'$ that restricts to a bijection between $S$ and $S'$.
In Section~\ref{delete vert sec}, we prove the following theorem.
\begin{theorem}\label{para factor}
Let $\eta:W\to W'$ be a surjective lattice homomorphism and let $J=\set{s\in S:\eta(s)\neq 1'}$.
Then $\eta$ factors as $\eta|_{W_J}\circ\eta_J$.
The map $\eta|_{W_J}$ (the restriction of $\eta$ to $W_J$) is a compressive homomorphism.
\end{theorem}

Parabolic homomorphisms and their associated congruences are well understood.
(See \cite[Section~6]{congruence}.)
The task, therefore, becomes to understand compressive homomorphisms between Coxeter groups.
To study compressive homomorphisms from $(W,S)$ to $(W',S')$, we may as well take $S'=S$ and require $\eta$ to restrict to the identity on $S$.
For each $r,s\in S$, let $m(r,s)$ be the order of $rs$ in $W$, and let $m'(r,s)$ be the order of $rs$ in $W'$.
Elementary considerations show that if $\eta:W\to W'$ is compressive, then $m'(r,s)\le m(r,s)$ for each pair $r,s\in S$.
(See Proposition~\ref{diagram facts}.)
Thus a compressive homomorphism corresponds to a diagram homomorphism that only erases edges from and/or reduces edge labels on the diagram of $W$.
More surprising, this property is sufficient to guarantee the existence of a compressive homomorphism.
The following theorem shows that Example~\ref{miraculous} is typical, rather than unusual.
Together with Theorem~\ref{para factor}, it implies Theorem~\ref{main} and adds additional detail.

\begin{theorem}\label{existence}
Suppose $(W,S)$ and $(W',S)$ are finite Coxeter systems.
Then there exists a compressive homomorphism from $W$ to $W'$, fixing $S$, if and only if $m'(r,s)\le m(r,s)$ for each pair $r,s\in S$.
If so, then the homomorphism can be chosen so that the associated congruence on $W$ is homogeneous of degree~$2$.
\end{theorem}

We will review the definition of homogeneous congruences in Section~\ref{shard sec}.
Informally, a homogeneous congruence of degree~$2$ is a congruence that is determined by contracting edges located in the order ideal of the weak order that describes the diagram of $W$, as in Figure~\ref{weakB3diagram}.
The situation is perhaps best appreciated by analogy:
Showing that $W'$ is isomorphic to the quotient of $W$ modulo a homogeneous congruence of degree~$2$ is analogous to finding that a graded ring $R$ is isomorphic to another graded ring $R'$ modulo an ideal generated by homogeneous elements of degree~$2$.
It should be noted, however, that lattice congruences are in general more complicated than ring congruences because the classes of a lattice congruence are not in general defined by an ideal.
See \mbox{\cite[Sections~II.3--4]{Birkhoff}}.

Given $(W,S)$ and $(W',S)$ with $m'(r,s)\le m(r,s)$ for each pair $r,s\in S$, there may be several homomorphisms from $W$ to $W'$ whose associated congruence is homogeneous of degree $2$.
There may also be several homomorphisms whose congruence is not homogeneous.
(In these non-homogeneous cases, the degree of the congruence is always $3$.)
These several possibilities are well-characterized, as we now explain.
Elementary considerations show that the restriction of a compressive homomorphism to any standard parabolic subgroup is still compressive.
(See Proposition~\ref{diagram facts}.)
It turns out that compressive homomorphisms of Coxeter groups are determined by their restrictions to rank-two standard parabolic subgroups.

\begin{theorem}\label{diagram uniqueness}
Let $(W,S)$ and $(W',S)$ be finite Coxeter systems with $m'(r,s)\le m(r,s)$ for each pair $r,s\in S$.
For each $\set{r,s}\subseteq S$, fix a surjective homomorphism $\eta_\set{r,s}$ from $W_{\set{r,s}}$ to $W'_{\set{r,s}}$ with $\eta_\set{r,s}(r)=r$ and $\eta_\set{r,s}(s)=s$.
Then 
there is at most one homomorphism $\eta:W\to W'$ such that the restriction of $\eta$ to $W_{\set{r,s}}$ equals $\eta_\set{r,s}$ for each pair $r,s\in S$.
\end{theorem}
As will be apparent in Section~\ref{dihedral sec}, for each pair $\set{r,s}\subseteq S$ with $r\neq s$, there are exactly $\binom{a}{b}^2$ choices of $\eta_\set{r,s}$, where $a=m(r,s)-2$ and $b=m(r,s)-m'(r,s)$.
In Example~\ref{miraculous}, there are four ways to choose all of the maps $\eta_\set{r,s}$:  
For both pairs $r,s$ with $m(r,s)\le 3$, we must choose the identity map.
For the pair $r,s$ with $m(r,s)=4$ and $m'(r,s)=3$, there are four choices, corresponding to the four ways to contract one ``left'' side edge and one ``right'' side edge in the octagonal interval of Figure~\ref{weakB3diagram}.a.
Theorem~\ref{diagram uniqueness} says, in particular, that in the example there are at most four homomorphisms that fix $S$ pointwise.

Combining Theorem~\ref{diagram uniqueness} with Theorem~\ref{para factor} leads immediately to the following more general statement.
\begin{corollary}\label{surjective uniqueness}
Let $W$ and $W'$ be finite Coxeter groups.
A surjective homomorphism from the weak order on $W$ to the weak order on $W'$ is determined by its restrictions to rank-two standard parabolic subgroups.
\end{corollary}

The statement of Theorem~\ref{diagram uniqueness} on the uniqueness of compressive homomorphisms, given their restrictions to rank-two standard parabolic subgroups, is remarkably close to being an existence and uniqueness theorem, in the sense that the phrase ``at most one'' can almost be replaced with ``exactly one.''
The only exceptions arise when $W$ has a standard parabolic subgroup of type $H_3$ such that the corresponding standard parabolic subgroup of $W'$ is of type $B_3$.
In particular, adding the hypothesis that $W$ and $W'$ are crystallographic turns Theorem~\ref{diagram uniqueness} into an existence and uniqueness theorem (stated as Theorem~\ref{existence uniqueness crys}).
Less generally, when $W$ and $W'$ are simply laced meaning that all edges in their diagrams are unlabeled), the existence and uniqueness theorem holds, and in fact this simply laced version of the theorem (Corollary~\ref {existence uniqueness simply}) has a uniform proof, given in Sections~\ref{erase edge sec}--\ref{shard sec}.
The remainder of the classification is proved, type-by-type in the classification of finite Coxeter groups, in Sections~\ref{dihedral sec}, \ref{Bn Sn+1 sec}, and \ref{exceptional sec}.

\begin{remark}\label{type-by-type}
It is disappointing that some of these proofs are not uniform.
However, as described above, the classification of surjective homomorphisms between weak orders on finite Coxeter groups is itself not uniform.
While nice things happen quite generally, the exceptions in some types suggest that uniform arguments probably don't exist outside of the simply-laced case (Corollary~\ref {existence uniqueness simply}).
In particular, although Theorem~\ref{existence uniqueness crys} is a uniform statement about the crystallographic case, there is no indication that the combinatorial lattice theory of the weak order detects the crystallographic case, so a uniform proof would be surprising indeed.
\end{remark}

In Section~\ref{camb sec}, we use the classification of surjective lattice homomorphisms between weak orders to classify surjective lattice homomorphisms between Cambrian lattices.
Cambrian lattices are quotients of the weak order modulo certain congruences called Cambrian congruences.
A Cambrian lattice can also be realized as a sublattice of the weak order consisting of sortable elements \cite{sortable,sort_camb,camb_fan}.
The significance of the Cambrian lattices begins with a collection of results, conjectured in~\cite{cambrian} and proved in~\cite{HLT,sortable,sort_camb,camb_fan}, which say that the Cambrian lattices and the related Cambrian fans encode the combinatorics and geometry of generalized associahedra of~\cite{ga}, which in turn provide a combinatorial model \cite{ga,ca2,camb_fan,framework} for cluster algebras of finite type.

The classification of surjective lattice homomorphisms between Cambrian lattices, given in Theorems~\ref{camb para factor}, \ref{camb exist unique}, and~\ref{camb diagram}, parallels the classification of surjective lattice homomorphisms between weak orders.
The main difference is that the Cambrian lattice results have uniform statements in terms of \newword{oriented diagram homomorphisms}.
(However,  our proofs rely on the non-uniform proofs given earlier for the weak order.)
An example of a compressive homomorphism between Cambrian lattices appears as Example~\ref{camb hom ex}, which continues Example~\ref{miraculous}.

Interesting geometric consequences are obtained by combining the results of this paper with \cite[Theorem~1.1]{con_app}.
The latter theorem states that every lattice congruence $\Theta$ on the weak order on $W$ defines a polyhedral fan $\F_\Theta$ that coarsens the fan $\F(W)$ defined by the reflecting hyperplanes of $W$.
The theorem also describes the interaction between the combinatorics/geometry of the fan $\F_\Theta$ and the combinatorics of the quotient lattice.
(In type A, the fan $\F_\Theta$ is known to be polytopal~\cite{PS} for any $\Theta$, but no general polytopality result is known in other types.)
The fact that a surjective lattice homomorphism $\eta:W\to W'$ exists whenever $m'(r,s)\le m(r,s)$ for each pair $r,s\in S$ leads to explicit constructions of a fan $\F_\Theta$ coarsening $\F(W)$ such that $\F_\Theta$ is combinatorially isomorphic to the fan $\F(W')$ defined by the reflecting hyperplanes of $W'$.
An example of this geometric point of view (corresponding to Example~\ref{miraculous}) appears as Example~\ref{B3 to A3 shard}.

Working along the same lines for surjective congruences between Cambrian lattices, we obtain refinement relationships between Cambrian fans (fans associated to Cambrian congruences).
We conclude the introduction by describing these refinement relationships in terms of dominance relationships between Cartan matrices.

A Cartan matrix $A=[a_{ij}]$ \newword{dominates} a Cartan matrix $A'=[a'_{ij}]$ if ${|a_{ij}|\ge |a'_{ij}|}$ for all $i$ and $j$.
The dominance relation on Cartan matrices implies that $m(r,s)\ge m'(r,s)$ for all $r,s\in S$ in the corresponding Weyl groups.
In the following proposition, $\Phi(A)$ is the full root system for $A$, including any imaginary roots.
In this paper, we only use the proposition in finite type, where there are no imaginary roots.
The proposition follows from known facts about Kac-Moody Lie algebras (see Section~\ref{crys case}), and has also been pointed out as \cite[Lemma~3.5]{Marquis}.

\begin{proposition}\label{dom subroot}
Suppose $A$ and $A'$ are symmetrizable Cartan matrices such that $A$ dominates~$A'$.
If $\Phi(A)$ and $\Phi(A')$ are both defined with respect to the same simple roots $\alpha_i$, then $\Phi(A)\supseteq\Phi(A')$ and $\Phi_+(A)\supseteq\Phi_+(A')$.
\end{proposition}

Proposition~\ref{dom subroot} may appear to be obviously false to someone who is familiar with root systems.
To clarify, we emphasize that defining both $\Phi(A)$ and $\Phi(A')$ with respect to the same simple roots means identifying the root space of $\Phi(A)$ with the root space of $\Phi(A')$ \emph{by identifying the bases $\set{\alpha_i}$}.
Thus we may restate the proposition as follows: The set of simple root coordinate vectors of roots in $\Phi(A')$ is a subset of the set of simple root coordinate vectors of roots in $\Phi(A)$.

The refinement result on Cambrian lattices is best expressed in terms of co-roots, so we rephrase Proposition~\ref{dom subroot} as Proposition~\ref{dom subcoroot}, which asserts a containment relation among dual root systems $\Phi\ck(A)$ and $\Phi\ck(A')$ when we identify the simple co-roots instead of the simple roots. 
The following result is proved by constructing (in Theorem~\ref{weak hom subroot}) the appropriate homomorphism from $W$ to $W'$ and using it to construct a homomorphism of Cambrian lattices.

\begin{theorem}\label{camb fan coarsen}
Suppose $A$ and $A'$ are Cartan matrices such that $A$ dominates $A'$ and suppose $W$ and $W'$ are the associated groups, both generated by the same set $S$.
Suppose $c$ and $c'$ are Coxeter elements of $W$ and $W'$ respectively that can be written as a product of the elements of $S$ in the same order.
Choose a root system $\Phi(A)$ and a root system $\Phi(A')$ so that the simple \emph{co}-roots are the same for the two root systems.
Construct the Cambrian fan for $(A,c)$ by coarsening the fan determined by the Coxeter arrangement for $\Phi(A)$ and construct the Cambrian fan for $(A',c')$ by coarsening the fan determined by the Coxeter  arrangement for $\Phi(A')$.
Then the Cambrian fan for $(A,c)$ refines the Cambrian fan for $(A',c')$.
Whereas the codimension-$1$ faces of the Cambrian fan for $(A,c)$ are orthogonal to co-roots (i.e.\ elements of $\Phi\ck(A)$), the Cambrian fan for $(A',c')$ is obtained by removing all codimension-$1$ faces orthogonal to elements of $\Phi\ck(A)\setminus\Phi\ck(A')$.  
\end{theorem}

As mentioned above and as explained in \cite[Section~5]{framework}, Cambrian fans provide a combinatorial model for cluster algebras of finite type.  
The cluster-algebraic consequences of Theorem~\ref{camb fan coarsen} are considered in \cite{dominance}.
Indeed, inspired by Theorem~\ref{camb fan coarsen}, the paper \cite{dominance} studies much more general cluster-algebraic phenomena related to dominance relations among matrices.

\begin{remark}\label{simion rem}
To the author's knowledge, the first appearance in the literature of a nontrivial surjective lattice homomorphism between finite Coxeter groups is a map found in Rodica Simion's paper \cite{Simion}.
(See Section~\ref{simion sec}.)
Simion's motivations were not lattice-theoretic, so she did not show that the map is a surjective lattice homomorphism.
However, she did prove several results that hint at lattice theory, including the fact that fibers of the map are intervals and that the \emph{order-theoretic} quotient of $B_n$ modulo the fibers of the map is isomorphic to $S_{n+1}$.
It was Simion's map that first alerted the author to the fact that interesting homomorphisms exist.
\end{remark}

\section{Deleting vertices}\label{delete vert sec}
In this section, we develop the most basic theory of surjective lattice homomorphisms between weak orders, leading to the proof of Theorem~\ref{para factor}, which factors a surjective homomorphism into a parabolic homomorphism and a compressive homomorphism.

We assume the standard background about Coxeter groups and the weak order, which is found, for example, in~\cite{Bj-Br}.
(For an exposition tailored to the point of view of this paper, see~\cite{regions10}.)
As we go, we introduce background on the combinatorics of homomorphisms and congruences of finite lattices.
Proofs of assertions not proved here are found in \cite[Section~9-5]{regions9}.

Let $(W,S)$ be a Coxeter system with identity element $1$.
The usual length function on $W$ is written $\ell$.
The pairwise orders of elements $r,s\in S$ are written $m(r,s)$.

The symbol $W$ will denote not only the group $W$ but also a partial order, the (right) weak order on the Coxeter group $W$.
This is the partial order on $W$ whose cover relations are of the form $w\covered ws$ for all $w\in W$ and $s\in S$ with $\ell(w)<\ell(ws)$.
The set $T$ of reflections of $W$ is $\set{wsw^{-1}:w\in W,s\in S}$.
The inversion set of and element $w\in W$ is $\inv(w)=\set{t\in T:\ell(tw)<\ell(w)}$.
An element is uniquely determined by its inversion set.
The weak order on $W$ corresponds to containment order on inversion sets.
The minimal element of $W$ is $1$ and the maximal element is $w_0$.
We have $w_0=\Join S$.

Given $J\subseteq S$, the standard parabolic subgroup generated by $J$ is written $W_J$.
This is, in particular, a lower interval in $W$.
The maximal element of $W_J$ is $w_0(J)$, which equals $\Join J$.
We need a second Coxeter system $(W',S')$, and we use the same notation for $W'$ as for $W$, with primes added to distinguish the groups.
As a first step, we prove the following basic facts:

\begin{proposition}\label{basic facts}
Let $\eta:W\to W'$ be a surjective lattice homomorphism.
Then
\begin{enumerate}
\item $\eta(1)=1'$ and $\eta(w_0)=w'_0$.
\item $S'\subseteq\eta(S)\subseteq (S'\cup\set{1'})$.
\item If $r$ and $s$ are distinct elements of $S$ with $\eta(r)=\eta(s)$, then $\eta(r)=\eta(s)=1'$.
\item If $J\subseteq S$, then $\eta$ restricts to a surjective homomorphism $W_J\to W'_{\eta(J)\setminus\set{1'}}$.
\item $m'(\eta(r),\eta(s))\le m(r,s)$ for each pair $r,s\in S$ with $\eta(r)\neq1'$ and $\eta(s)\neq1'$.
\end{enumerate}
\end{proposition}
\begin{proof}
A surjective homomorphism of finite lattices takes the minimal element to the minimal element and the maximal element to the maximal element, so (1) holds.

Suppose $s\in S$ has $\eta(s)\not\in S'\cup\set{1'}$.
Then there exists $s'\in S'$ such that $\eta(s)>s'$.
Since $\eta$ is surjective, there exists $w\in W$ such that $\eta(w)=s'$.
But $\eta$ is order-preserving and $s\le (s\join w)$, so $\eta(s)\le\eta(s\join w)=\eta(s)\join s'=s'$,  and this contradiction shows that $\eta(S)\subseteq S'\cup\set{1'}$.
We have $\Join\eta(S)=\eta(w_0)=w'_0$.
But $\Join J'<w'_0$ for any proper subset $J'$ of $S'$, so $S'\subseteq\eta(S)$, and we have proved (2).

To prove (3), let $r$ and $s$ be distinct elements of $S$ with $\eta(r)=\eta(s)$.
Then $1'=\eta(1)=\eta(r\meet s)=\eta(r)\meet\eta(s)=\eta(r)$.

Applying $\eta$ to $w_0(J)$ yields $\eta(\Join J)$, which equals $\Join_{s\in J}\eta(s)=w'_0(\eta(J)\setminus\set{1'})$.
Since $W_J$ is the interval $[1,w_0(J)]$ and $\eta$ is order-preserving, $\eta(W_J)\subseteq[1,w'_0(\eta(J)\setminus\set{1'})]=W'_{\eta(J)\setminus\set{1'}}$.
If $w'\in W'_{\eta(J)\setminus\set{1'}}$, then since $\eta$ is surjective, there exists $w\in W$ such that $\eta(w)=w'$.
Then $w_0(J)\meet w$ is in $W_J$, and $\eta(w_0(J)\meet w)=w'_0(\eta(J)\setminus\set{1'})\meet w'=w'$.
We have proved (4).

If $\eta(r)\neq1'$ and $\eta(s)\neq1'$, then (3) says that $\eta(r)\neq\eta(s)$ and (4) says that $\eta$ restricts to a surjective homomorphism from the rank-two standard parabolic subgroup $W_{\set{r,s}}$ to the rank-two standard parabolic subgroup $W'_{\set{\eta(r),\eta(s)}}$.
Thus $|W_{\set{r,s}}|\ge|W'_{\set{\eta(r),\eta(s)}}|$.
This is equivalent to (5).
\end{proof}

Given $J\subseteq S$, for any $w\in W$, there is a unique factorization $w=w_J\,\cdot \!\phantom{.}^J\! w$ that maximizes 
$\ell(w_J)$ subject to the constraints $\ell(w_J)+\ell(\!\!\phantom{.}^J\! w)=\ell(w)$ and $w_J\in W_J$.
The element $w_J$ is also the unique element of $W$ (and the unique element of $W_J$) whose inversion set is $\inv(w)\cap W_J$.
Let $\eta_J:W\to W_J$ be the map sending $w$ to $w_J$.
We call $\eta_J$ a \newword{parabolic homomorphism}.
We now illustrate parabolic homomorphisms in terms of the usual combinatorial representations for types $A_n$ and $B_n$ and in terms of the usual geometric representation for type $H_3$.

\begin{example}\label{An para}
We realize a Coxeter group of type $A_n$ in the usual way as the symmetric group $S_{n+1}$ of permutations of $\set{1,\ldots,n+1}$, with simple reflections $S=\set{s_i:i=1,\ldots n}$, where each $s_i$ is the adjacent transposition $(i\,\,\,\,i\!+\!1)$.
We write permutations $\pi$ in one-line notation $\pi=\pi_1\pi_2\cdots\pi_{n+1}$ with each $\pi_i$ standing for $\pi(i)$.
Choose some $k\in\set{1,\ldots,n}$ and let $J=\set{s_1,\ldots,s_{k-1}}\cup\set{s_{k+1},\ldots,s_n}$.
The map $\eta_J$ corresponds to deleting the vertex $s_k$ from the Coxeter diagram for $A_n$, thus splitting it into components, one of type $A_{k-1}$ and one of type $A_{n-k}$.
The map $\eta_J$ takes $\pi$ to $(\sigma,\tau)\in S_k\times S_{n+1-k}$, where $\sigma$ is the permutation of $\set{1,2,\ldots,k}$ given by deleting from the sequence $\pi_1\pi_2\cdots\pi_{n+1}$ all values greater than $k$ and $\tau$ is the permutation of $\set{1,2,\ldots,n+1-k}$ given by deleting from $\pi_1\pi_2\cdots\pi_{n+1}$ all values less than $k+1$ and subtracting $k$ from each value.
For example, if $n=7$ and $k=3$, then $\eta_J(58371426)=(312,25413)$.
The map $\eta_J$ is similarly described for more general~$J$.
\end{example}

\begin{example}\label{Bn para}
As usual, we realize a Coxeter group of type $B_n$ as the group of \newword{signed permutations}.
These are permutations $\pi$ of $\set{\pm1,\pm2,\ldots,\pm n}$ with $\pi(-i)=-\pi(i)$ for all $i$.
The simple generators are the permutations $s_0=(1\,\,\,-\!1)$ and $s_i=(-i\!-\!1\,\,\,\,-i)(i\,\,\,\,i\!+\!1)$ for $i=1,\ldots,n-1$.
A signed permutation $\pi$ is determined by its one-line notation $\pi=\pi_1\pi_2\cdots\pi_n$, where each $\pi_i$ again stands for $\pi(i)$.
For $J=\set{s_0,\ldots,s_{k-1}}\cup\set{s_{k+1},\ldots,s_{n-1}}$, the map $\eta_J$ corresponds to deleting the vertex $s_k$, splitting the diagram of $B_n$ into components, of types $B_k$ and $A_{n-k-1}$.
The map $\eta_J$ takes $\pi$ to $(\sigma,\tau)$, where $\sigma$ is the signed permutation whose one-line notation is the restriction of the sequence $\pi_1\pi_2\cdots\pi_n$ to values with absolute value less than $k+1$ and $\tau$ is the permutation given by restricting the sequence $(-\pi_n)(-\pi_{n-1})\cdots(-\pi_1)\pi_1\cdots\pi_{n-1}\pi_n$ to positive values greater than $k$, and then subtracting $k$ from each value.
For example, if $n=8$, $k=4$, and $\pi=(-4)(-2)71(-8)(-6)5(-3)$, then $\eta_J(\pi)=((-4)(-2)1(-3),2431)$.
\end{example}


\begin{example}\label{H3 para}
For $W$ of type $H_3$, we give a geometric, rather than combinatorial, description of parabolic homomorphisms.
Figure~\ref{H3 para fig}.a shows the reflecting planes of a reflection representation of $W$.
\begin{figure}
\begin{tabular}{ccc}
\includegraphics{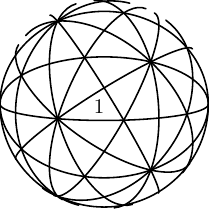}&&\includegraphics{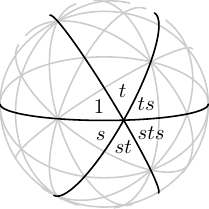}\\
(a)&&(b)\\[8pt]
\includegraphics{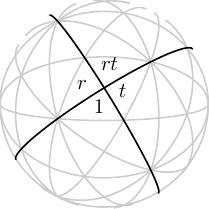}&&\includegraphics{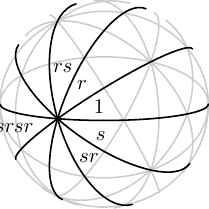}\\
(c)&&(d)
\end{tabular}
\caption{
a:  Reflecting planes for $W$ of type $H_3$. 
b:  Reflecting planes for $W_{\set{s,t}}$.
c:  Reflecting planes for $W_{\set{r,t}}$.
d:  Reflecting planes for $W_{\set{r,s}}$.
}
\label{H3 para fig}
\end{figure}
Each plane is represented by its intersection with a unit sphere about the origin.
The sphere is considered to be opaque, so that we only see the side of the sphere that is closest to us.
The spherical triangles traced out on the sphere (including those on the back of the sphere that we can't see) are in bijection with the elements of $W$.
Specifically, the triangle corresponding to $1$ is marked, and each $w\in W$ corresponds to the image of the triangle marked $1$ under the action of $w$.
Taking $S=\set{r,s,t}$ with $m(r,s)=5$, $m(s,t)=3$ and $m(r,t)=2$, the reflecting planes of the reflections $S$ are the three planes that bound the triangle marked $1$.
The plane corresponding to $s$ is nearly horizontal in the picture and the plane for $r$ intersects the plane for $s$ at the left of the triangle marked $1$.

The parabolic congruence that deletes the vertex $r$ can be seen geometrically in Figure~\ref{H3 para fig}.b, which shows the reflecting planes for the standard parabolic subgroup $W_{\set{s,t}}$.
The sectors cut out by these planes correspond to the elements of $W_{\set{s,t}}$, as shown in the picture.
The parabolic congruence maps an element $w\in W$, corresponding to a triangle $\mathcal{T}$, to the element of $W_{\set{s,t}}$ labeling the sector containing $\mathcal{T}$.
Similar pictures, Figures~\ref{H3 para fig}.c--d, describe the parabolic congruences deleting the vertices $s$ and $t$ respectively.
(Some labels are left out of the representation of $W_{\set{r,s}}$ in Figure~\ref{H3 para fig}.d.)
\end{example}

The following theorem is a concatenation of \cite[Proposition~6.3]{congruence} and \cite[Corollary~6.10]{congruence}.
The fact that $\eta_J$ is a lattice homomorphism was also established in~\cite{Jed}.

\begin{theorem}\label{para cong}
If $J\subseteq S$, then $\eta_J$ is a surjective homomorphism.
Its fibers constitute the finest lattice congruence on $W$ with $1\equiv s$ for all $s\in S\setminus J$.
\end{theorem}

Proposition~\ref{basic facts} and Theorem~\ref{para cong} lead to the proof of Theorem~\ref{para factor}.
To give the proof, we need the following basic observation about a congruence $\Theta$ on a finite lattice $L$:
Congruence classes are intervals, and the quotient $L/\Theta$ is isomorphic to the subposet of $L$ induced by the elements of $L$ that are the bottom elements of congruence classes.
Let $\Theta_J$ be the congruence whose classes are the fibers of $\eta_J$.
An element $w\in W$ is at the bottom of its $\Theta_J$-class if and only if $w\in W_J$.
Thus $W/\Theta_J$ is isomorphic to $W_J$.

\begin{proof}[Proof of Theorem~\ref{para factor}]
If $\Theta$ is the lattice congruence on $W$ whose congruence classes are the fibers of $\eta$, then Proposition~\ref{basic facts}(1) says that $1\equiv s$ for all $s\in S\setminus J$.
Theorem~\ref{para cong} says that the congruence $\Theta_J$, determined by the fibers of the homomorphism $\eta_J$, is a refinement of the congruence $\Theta$.
Thus $\eta$ factors as $\eta=\eta'\circ\nu$, where $\nu:W\to W/\Theta_J$ is the natural map.
The map $\eta'$ maps a $\Theta_J$-class to $\eta(w)$ where $w$ is any element of the $\Theta_J$-class.
However, each $\Theta_J$-class contains a unique element of $W_J$, so we can replace $\nu$ with the map $\eta_J$ and replace $\eta'$ with the restriction of $\eta$ to $W_J$.
By Proposition~\ref{basic facts}(2--3), $\eta|_{W_J}$ restricts to a bijection from $J$ to~$S'$. 
\end{proof}

Theorem~\ref{para factor} reduces the problem of classifying surjective homomorphisms $\eta$ between weak orders to the special case of classifying compressive homomorphisms between weak orders.
As in the introduction, we may as well restrict to the case where $S=S'$ and $\eta$ restricts to the identity on $S$.
For convenience, we rewrite some of the assertions of Proposition~\ref{basic facts} in the case where $\eta$ is compressive:

\begin{proposition}\label{diagram facts}
Let $(W,S)$ and $(W',S)$ be finite Coxeter systems.
If $\eta:W\to W'$ is a compressive homomorphism fixing $S$ pointwise, then 
\begin{enumerate}
\item If $J\subseteq S$, then $\eta$ restricts to a surjective homomorphism $W_J\to W'_J$.
\item $m'(r,s)\le m(r,s)$ for each pair $r,s\in S$.
\end{enumerate}
\end{proposition}

Theorem~\ref{para factor} shows that surjective homomorphisms correspond to deleting vertices of the diagram and then applying a compressive homomorphism, and the second assertion of Proposition~\ref{diagram facts} is an additional step towards Theorem~\ref{main}:
It shows that the compressive homomorphism decreases edge labels and/or erases edges.
The remainder of the paper is devoted to classifying compressive homomorphisms between finite Coxeter groups.

\section{Erasing edges}\label{erase edge sec}
In this section, we begin the classification of compressive homomorphisms by considering the simplest case, the case of compressive homomorphisms that erase edges but otherwise do not decrease edge labels.
That is, we consider the case where, for each $r,s\in S$, either $m'(r,s)=m(r,s)$ or $m'(r,s)=2$.

Recall that the diagram of a finite Coxeter group is a forest (a graph without cycles), so removing any edge breaks a connected component of the diagram into two pieces.
Given a set $E$ of edges of the diagram for $W$, write $S$ as a disjoint union of sets $J_1,J_2,\ldots,J_k$ such that each set $J_i$ is the vertex set of a connected component of the graph obtained by deleting the edges $E$ from the diagram.
Define $\eta_E$ to be the map from $W$ to $W_{J_1}\times W_{J_2}\times\cdots\times W_{J_k}$ that sends $w\in W$ to $(w_{J_1},w_{J_2},\ldots,w_{J_k})$.
We call $\eta_E$ an \newword{edge-erasing homomorphism}.
Beginning in this section and finishing in Section~\ref{shard sec}, we prove the following theorem.

\begin{theorem}\label{edge factor}
Let $\eta:W\to W'$ be a compressive homomorphism fixing $S$ pointwise.
Let $E$ be any set of edges in the diagram of $W$ such that each edge $r$---$s$ in $E$ has $m'(r,s)=2$.
Let $J_1,J_2,\ldots,J_k$ be the vertex sets of the connected components of the graph obtained by deleting the edges $E$ from the diagram for $W$.
In particular, $W'\cong W'_{J_1}\times W'_{J_2}\times\cdots\times W'_{J_k}$.
Then $\eta$ factors as $\eta'\circ\eta_E$, where ${\eta':W_{J_1}\times W_{J_2}\times\cdots\times W_{J_k}}\to W'$ is the compressive homomorphism with $\eta'(w_1,\ldots, w_k)=(\eta(w_1),\ldots,\eta(w_k))$.
\end{theorem}

The proof of Theorem~\ref{edge factor} is similar to the proof of Theorem~\ref{para factor}.
We characterize the congruence associated to $\eta_E$ as the finest congruence containing certain equivalences and conclude that any homomorphism that erases the edges $E$ factors through~$\eta_E$.

Let $r$ and $s$ be distinct elements of $S$ and let $m=m(r,s)$.
Suppose $r$ and $s$ form an edge in $E$, or in other words suppose $m\ge 3$.
Then the standard parabolic subgroup $W_{\set{r,s}}$ is the lower interval $[1,w_0(\set{r,s})]$, consisting of two chains:  $1\covered r\covered rs\covered rsr\covered\cdots\covered w_0(\set{r,s})$ and $1\covered s\covered sr\covered srs\covered\cdots\covered w_0(\set{r,s})$.
We define $\alt_k(r,s)$ to be the word with $k$ letters, starting with $r$ and then alternating $s$, $r$, $s$, etc.
Thus the two elements covered by $w_0(\set{r,s})$ are $\alt_{m-1}(r,s)$ and $\alt_{m-1}(s,r)$.
The key to the proof of Theorem~\ref{edge factor} is the following theorem.

\begin{theorem}\label{edge cong}
If $E$ is a set of edges of the diagram of $W$, then $\eta_E$ is a compressive homomorphism.
Its fibers constitute the finest congruence with $r\equiv\alt_{m(r,s)-1}(r,s)$ and $s\equiv\alt_{m(r,s)-1}(s,r)$ for all edges $r$---$s$ in $E$.
\end{theorem}

Each map $\eta_{J_i}$ is a lattice homomorphism by Theorem~\ref{para cong}, and we easily conclude that $\eta_E$ is a homomorphism as well.
To see that $\eta_E$ is surjective, consider $(w_1,\ldots,w_k)\in W_{J_1}\times W_{J_2}\times\cdots\times W_{J_k}$.
Then each $w_i$ is also an element of $W$.
Furthermore, $\eta_{J_i}(w_j)=1$ if $j\neq i$.
Thus the fact that $\eta_E$ is a lattice homomorphism implies that $\eta_E(\Join_{i=1}^kw_i)=(w_1,\ldots,w_k)$.
We have established the first assertion of Theorem~\ref{edge cong}.
The second assertion requires more background on lattice congruences of the weak order.
This background and the proof of the second assertion is given in Section~\ref{shard sec}.
For now, we show how Theorem~\ref{edge cong} is used to prove Theorem~\ref{edge factor}.

\begin{proof}[Proof of Theorem~\ref{edge factor}, given Theorem~\ref{edge cong}]
Let $r$---$s$ be an edge in $E$.  
In the Coxeter group $W'$, $r\join s=rs$.
By Proposition~\ref{diagram facts}, the restriction of $\eta$ to the interval $[1,w_0(\set{r,s})]$ in $W$ is a surjective lattice homomorphism to the interval $[1',rs]$ in~$W'$.
By hypothesis, $\eta$ fixes $r$ and $s$.
Since $\eta$ is order-preserving, $\eta(\alt_{m(r,s)-1}(r,s))$ is either $r$ or $rs$.
But if $\eta(\alt_{m(r,s)-1}(r,s))=rs$, then $\eta(s\meet\alt_{m(r,s)-1}(r,s))=rs\meet s=s$.
But $s\meet\alt_{m(r,s)-1}(r,s)=1$, and $\eta(1)=1'$, so we conclude that $\eta(\alt_{m(r,s)-1}(r,s))=r$.
Similarly, $\eta(\alt_{m(r,s)-1}(s,r))=s$.
Thus if $\Theta$ is the lattice congruence on $W$ whose congruence classes are the fibers of $\eta$, then $r\equiv\alt_{m(r,s)-1}(r,s)$ and $s\equiv\alt_{m(r,s)-1}(s,r)$ modulo $\Theta$.
Let $\Theta_E$ be the lattice congruence whose classes are the fibers of $\eta_E$.
Theorem~\ref{edge cong} says that the congruence $\Theta_E$ is a refinement of the congruence $\Theta$.

Thus $\eta$ factors through the natural map $\nu:W\to W/\Theta_E$.
Equivalently, we can factor $\eta$ as $\eta'\circ\eta_E$, where $\eta'$ maps $(w_1,\ldots,w_k)\in W_{J_1}\times\cdots\times W_{J_k}$ to $\eta(w)$, where $w$ is any element in the $\eta$-fiber of $(w_1,\ldots,w_k)$.
Specifically, we can take $\eta'(w_1,\ldots,w_k)=\eta(\Join_{i=1}^kw_i)$.
Since $\eta$ is a homomorphism, the latter is $\Join_{i=1}^k\eta(w_i)$, which equals $(\eta(w_1),\ldots,\eta(w_k))$ because $W'\cong W'_{J_1}\times W'_{J_2}\times\cdots\times W'_{J_k}$.
It now follows from Proposition~\ref{diagram facts} that $\eta'$ is a surjective homomorphism.
\end{proof}

Theorem~\ref{edge factor} has the following immediate corollary.

\begin{corollary}\label{existence uniqueness simply}
Let $(W,S)$ and $(W',S)$ be finite, simply laced finite Coxeter systems such that the diagram of $W'$ is obtained from the diagram of $W$ by erasing a set $E$ of edges.
Then $\eta_E$ is the unique compressive homomorphism from $W$ to $W'$ fixing $S$ pointwise.
\end{corollary}

We conclude this section with some examples of edge-erasing homomorphisms.

\begin{example}\label{An erase edge}
We describe edge-erasing homomorphisms from $W$ of type $A_n$ in terms of the combinatorial realization described in Example~\ref{An para}.
If $E=\set{s_{k}\text{---}s_{k+1}}$, the edge-erasing homomorphism $\eta_E$ maps $\pi\in S_{n+1}$ to $(\sigma,\tau)\in S_{k+1}\times S_{n-k+1}$, where $\sigma$ is the restriction of the sequence $\pi_1\pi_2\cdots\pi_{n+1}$ to values $\le k+1$ and $\tau$ is obtained by restricting $\pi_1\pi_2\cdots\pi_{n+1}$ to values $\ge k+1$ and then subtracting $k$ from each value.
For example, if $n=7$ and $k=3$, then $\eta_E(58371426)=(3142,25413)$.
\end{example}

\begin{example}\label{Bn erase edge}
The description for type $B_n$ is similar.
If $E=\set{s_{k-1}\text{---}s_{k}}$, the edge-erasing homomorphism $\eta_E$ maps a signed permutation $\pi$ to $(\sigma,\tau)\in B_k\times S_{n-k+1}$, where $\sigma$ is the restriction of the sequence $\pi_1\pi_2\cdots\pi_{n+1}$ to entries $\pi_i$ with ${|\pi_i|\le k}$ and $\tau$ is obtained by restricting $(-\pi_n)(-\pi_{n-1})\cdots(-\pi_1)\pi_1\cdots\pi_{n-1}\pi_n$ to values $\ge k$ and then subtracting $k-1$ from each value.
For example, if $n=8$, $k=4$, and $\pi=(-4)(-2)71(-8)(-6)5(-3)$, then $\eta(\pi)=((-4)(-2)1(-3),35142)$.
\end{example}

\begin{example}\label{H3 erase edge}
We describe edge-erasing homomorphisms from $W$ of type $H_3$ in the geometric context introduced in Example~\ref{H3 para}.
Figure~\ref{H3 edge fig}.a--b represent the two edge-erasing homomorphisms.
\begin{figure}
\begin{tabular}{ccc}
\includegraphics{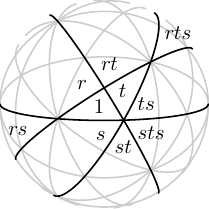}&&\includegraphics{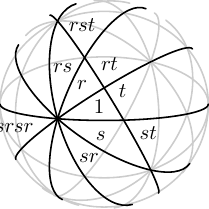}\\
(a)&&(b)\\
\end{tabular}
\caption{
a:  The homomorphism that erases the edge $r$---$s$ in type $H_3$. 
b:  The homomorphism that erases the edge $s$---$t$ in type~$H_3$.
}
\label{H3 edge fig}
\end{figure}  
Thus the homomorphism erasing $r$---$s$ maps each element $w\in W$, corresponding to a triangle $\mathcal{T}$, to the element of $W_{\set{r}}\times W_{\set{s,t}}$ labeling the region in Figure~\ref{H3 edge fig}.a containing $\mathcal{T}$.
Figure~\ref{H3 edge fig}.b represents the homomorphism erasing $s$---$t$ similarly.
In each picture, some labels are omitted or belong on the invisible side of the sphere.
\end{example}

\section{Lattice congruences of the weak order}\label{shard sec}
In this section, we quote results that give us the tools to prove Theorem~\ref{edge cong} and to complete the classification of compressive homomorphisms between finite Coxeter groups.
We prove Theorem~\ref{edge cong} at the end of this section and complete the classification in later sections.
For any results that are stated here without proof or citation, proofs can be found in \cite{congruence} and/or \cite[Section~9-5]{regions9}.

We begin with more details on congruences on a finite lattice $L$.
Recall from Section~\ref{intro} that a congruence $\Theta$ on $L$ is uniquely determined by the set of edges \newword{contracted} by $\Theta$ (the set of edges $x\covered y$ such that $x\equiv y$ modulo $\Theta$).
In fact, $\Theta$ is uniquely determined by a smaller amount of information.
An element $j$ of $L$ is \newword{join-irreducible} if it covers exactly one element $j_*$.
We say that $\Theta$ \newword{contracts} the join-irreducible element $j$ if $j\equiv j_*$ modulo $\Theta$.
The congruence $\Theta$ is determined by the set of join-irreducible elements that $\Theta$ contracts. 

The set $\Con(L)$ of all congruences on $L$ is a sublattice of the lattice of set partitions of $L$. 
In fact, $\Con(L)$ is a distributive lattice. 
We write $\Irr(\Con(L))$ for the set of join-irreducible congruences on $\Con(L)$.
By the Fundamental Theorem of Finite Distributive Lattices (see e.g.\ \cite[Theorem~3.4.1]{EC1}), $\Con(L)$ is isomorphic to inclusion order on the set of order ideals in $\Irr(\Con(L))$.

The weak order on a finite Coxeter group $W$ has a special property called \newword{congruence uniformity}, or sometimes called \newword{boundedness}.
(This was first proved in \cite[Theorem~6]{bounded}.  See also \cite[Theorem~27]{hyperplane}.)
The definition of congruence uniformity is not necessary here, but congruence uniformity means in particular that the join-irreducible congruences on $W$ (i.e.\ the join-irreducible elements of $\Con(L)$) are in bijection with the join-irreducible elements of $W$ itself. 
The join-irreducible elements of $W$ are the elements $j$ such that there exists a unique $s\in S$ with $\ell(js)<\ell(j)$.
For each join-irreducible element $j$ of $W$, the corresponding join-irreducible element $\Cg(j)$ of $\Con(W)$ is the unique finest congruence that contracts~$j$.
Thus $\Irr(\Con(W))$ can be thought of as a partial order on the join-irreducible elements of~$W$.
Congruences on $W$ correspond to order ideals in $\Irr(\Con(W))$.
Given a congruence $\Theta$ with corresponding order ideal $I$, the poset $\Irr(\Con(W/\Theta))$ is isomorphic to the induced subposet of $\Irr(\Con(W))$ obtained by deleting the elements~of~$I$.

The \newword{support} of an element $w$ of $W$ is the unique smallest subset $J$ of $S$ such that $w$ is in the standard parabolic subgroup $W_J$.
Given a join-irreducible element $j\in W$, the \newword{degree} of $j$ is the size of the support of $j$.
Given a set $\set{j_1,\ldots,j_k}$ of join-irreducible elements of $W$, there exists a unique finest congruence contracting all join-irreducible elements in the set.
This congruence is the join $\Cg(j_1)\join\cdots\join\Cg(j_k)$ in the congruence lattice $\Con(W)$, or equivalently, the congruence that contracts a join-irreducible element $j$ if and only if $j$ is in the ideal in $\Irr(\Con(W))$ generated by $\set{j_1,\ldots,j_k}$.
We call this the \newword{congruence generated by} $\set{j_1,\ldots,j_k}$.
A congruence on $W$ is \newword{homogeneous of degree $d$} if it is generated by a set of join-irreducible elements of degree $d$.
Abusing terminology slightly, we call a surjective homomorphism \newword{homogeneous of degree $d$} if its corresponding congruence is.

We restate Theorems~\ref{para cong} and~\ref{edge cong} with this new terminology.
(Recall that Theorem~\ref{para cong} was proven in \cite{congruence}.
We will prove Theorem~\ref{edge cong} below.)

\begin{theorem}\label{para cong ji}
If $J\subseteq S$, then $\eta_J$ is a surjective homogeneous homomorphism of degree $1$.
Its fibers constitute the congruence generated by $\set{s:s\in S\setminus J}$.
\end{theorem}

\begin{theorem}\label{edge cong ji}
If $E$ is a set of edges of the diagram of $W$, then $\eta_E$ is a compressive homogeneous homomorphism of degree $2$.
Its fibers constitute the congruence generated by $\set{\alt_k(r,s):\set{r,s}\in E\text{ and }k=2,\ldots,m(r,s)-1}$.
\end{theorem}
We emphasize that for each set $\set{r,s}$ in $E$ and each $k\in\set{2,\ldots,m(r,s)-1}$, both $\alt_k(r,s)$ and $\alt_k(s,r)$ are in the generating set described in Theorem~\ref{edge cong ji}.

Lattice congruences on the weak order on a finite Coxeter group $W$ are closely tied to the geometry of the reflection representation of $W$.
Let $\Phi$ be a root system associated to $W$ with simple roots $\Pi$.
For each simple reflection $s\in S$, let $\alpha_s$ be the associated simple root.
For each reflection $t\in T$, let $\beta_t$ be the positive root associated to $t$ and let $H_t$ be the reflecting hyperplane for $t$.
A point $x$ is \newword{below} $H_t$ if the inner product of $x$ with $\beta_t$ is nonnegative.
A set of points is below $H_t$ if each of the points is.
Points and sets are \newword{above} $H_t$ if the inner products are nonpositive.

The set $\A=\set{H_t:t\in T}$ is the Coxeter arrangement associated to $W$.
The arrangement $\A$ cuts space into \newword{regions}, which are in bijection with the elements of~$W$.
Specifically, the identity element of $W$ corresponds to the region $D$ that is below every hyperplane of $\A$, and an element $w$ corresponds to the region $wD$.

A subset $\A'$ of $\A$ is called a \newword{rank-two subarrangement} if $|\A'|>1$ and if there is some codimension-2 subspace $U$ such that $\A'=\set{H\in\A:H\supset U}$.
The subarrangement $\A'$ cuts space into $2|\A'|$ regions.
Exactly one of these regions, $D'$ is below all of the hyperplanes in $\A'$, and two hyperplanes in $\A'$ are facet-defining hyperplanes of $D'$.  
These two hyperplanes are called the \newword{basic hyperplanes} of~$\A'$.

We define a \newword{cutting relation} on the hyperplanes of $\A$ as follows:
Given distinct hyperplanes $H, H'\in\A$, let $\A'$ be the rank-two subarrangement containing $H$ and~$H'$.
Then $H$ \newword{cuts} $H'$ if~$H$ \textbf{is} a basic hyperplane of $\A'$ and $H'$ is \textbf{not} a basic hyperplane of $\A'$.
For each $H\in\A$, remove from~$H$ all points contained in hyperplanes of~$\A$ that cut~$H$.
The remaining set of points may be disconnected; the closures of the connected components are called the \newword{shards} in~$H$.
The set of shards of $\A$ is the union, over hyperplanes $H\in\A$, of the set of shards in $H$.
For each shard $\Sigma$, we write $H(\Sigma)$ for the hyperplane containing $\Sigma$.

\begin{example}\label{i25 shards}
When $W$ is a dihedral Coxeter group, the only rank-two subarrangement of $\A$ is $\A$ itself.
Figure~\ref{i25}.a shows the reflection representation of a Coxeter group of type $I_2(5)$ with $S=\set{r,s}$.
Figure~\ref{i25}.b shows the associated shards.
\begin{figure}
\begin{tabular}{ccc}
\scalebox{1}{\includegraphics{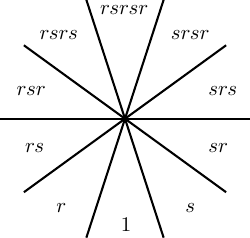}}&&\scalebox{1}{\includegraphics{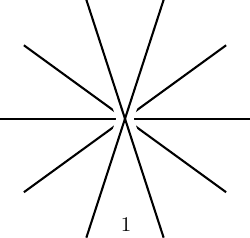}}\\
(a)&&(b)
\end{tabular}
\caption{(a):  A Coxeter group of type $I_2(5)$.
(b):  The associated shards}
\label{i25}
\end{figure}
Each of the shards contains the origin, but to make the picture legible, those shards that don't continue through the origin are drawn with an offset from the origin.
\end{example}

\begin{example}\label{B3 shard ex}
Figure~\ref{B3shards} depicts the shards in the Coxeter arrangement of type $B_3$.
\begin{figure}
\centerline{\scalebox{0.98}{\includegraphics{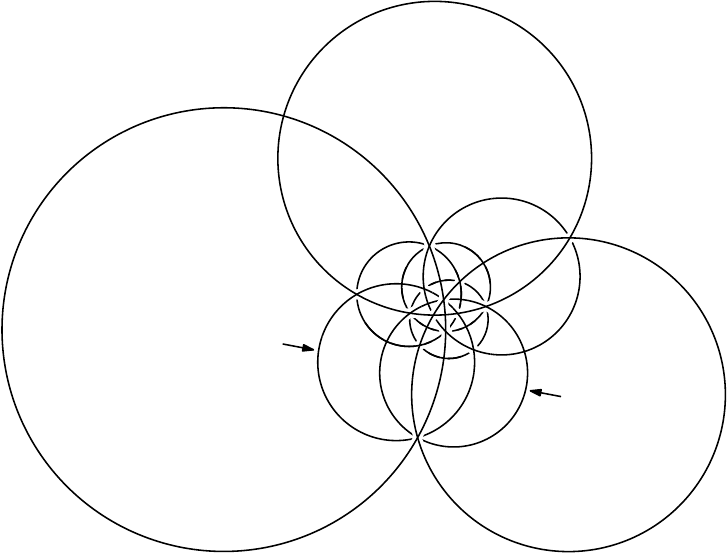}}}
\caption{The shards in a Coxeter group of type $B_3$}
\label{B3shards}
\end{figure}
Here, the arrangement $\A$ is a collection of nine planes in $\reals^3$.
The shards are two-dimensional cones contained in these planes.
To capture the picture in the plane, we consider the intersection of $\A$ with a sphere at the origin.
This intersection is an arrangement on nine great circles of the sphere.
Each shard, intersected with the sphere, is either an entire great circle or an arc of a great circle.
The figure shows these intersections under a stereographic projection from the sphere to the plane.
The region $D$ is the small triangle that is inside all of the circles.
As in Figure~\ref{i25}, where shards intersect, those shards that do not continue through the intersection are shown with an offset from the intersection.
Two of the shards are distinguished by arrows.
The significance of these two shards will be explained in Example~\ref{B3 to A3 shard}.
\end{example}

The shards of $\A$ are in one-to-one correspondence with the join-irreducible elements of $W$.
For each shard $\Sigma$, an \newword{upper element} of $\Sigma$ is an element $w\in W$ such that the region $wD$ is above $H(\Sigma)$ and intersects $\Sigma$ in codimension 1.
The set $U(\Sigma)$ of upper elements of $\Sigma$ contains exactly one element $j(\Sigma)$ that is join-irreducible in~$W$.
The element $j(\Sigma)$ is the unique minimal element of $U(\Sigma)$ in the weak order.
This is a bijection from shards to join-irreducible elements.  
The inverse map sends a join-irreducible element $j$ to $\Sigma(j)$, the shard that contains $jD\cap(j_*D)$.

Shards in $\A$ correspond to certain collections of edges in the Hasse diagram of the weak order on $W$.
Specifically, a shard $\Sigma$ corresponds to the set of all edges $x\covered y$ such that $(xD\cap yD)\subseteq\Sigma$.
For each shard $\Sigma$, a congruence $\Theta$ either contracts none of the edges associated to $\Sigma$ or contracts all of the edges associated to $\Sigma$.
(See \cite[Proposition~6.6]{shardint}.)
If $\Theta$ contracts all of the edges associated to $\Sigma$, then we say that $\Theta$ \newword{removes} $\Sigma$.
In particular, $\Theta$ removes a shard $\Sigma$ if and only if it contracts the join-irreducible element $j(\Sigma)$.
For any congruence $\Theta$, the set of shards not removed by $\Theta$ decomposes space into a fan \cite[Theorem~5.1]{con_app} that is a coarsening of the fan defined by the hyperplanes $\A$.

We now define a directed graph on shards, called the \newword{shard digraph}.
Given two shards~$\Sigma$ and $\Sigma'$, say $\Sigma\to\Sigma'$ if $H(\Sigma)$ cuts $H(\Sigma')$ and $\Sigma\cap\Sigma'$ has codimension~2.
The digraph thus defined on shards is called the \newword{shard digraph}.
This digraph is acyclic\footnote{Shards are often considered in a more general context of simplicial hyperplane arrangements.  In this broader setting, the shard digraph need not be acyclic.  See \cite[Figure~5]{hyperplane}.}
and we call its transitive closure the \newword{shard poset}.
The bijection $\Sigma\mapsto j(\Sigma)$ from shards to join-irreducible elements is an isomorphism from the shard poset to the poset $\Irr(\Con(W))$, thought of as a partial order on join-irreducible elements of $W$.
Thus, given any congruence $\Theta$, the set of shards removed by $\Theta$ is an order ideal in the shard poset.
Conversely, for any set of shards forming an order ideal in the shard poset, there is a congruence removing exactly that set of shards.

We use this correspondence to reuse shard terminology for join-irreducible elements and vice versa.
So, for example, we talk about the degree of a shard (meaning the degree of the corresponding join-irreducible element), etc.

\begin{example}\label{B3 to A3 shard}
Recall that Example~\ref{miraculous} started by choosing a pair of edges in the weak order on a Coxeter group of type $B_3$.
The finest congruence contracting the two edges was calculated, and the quotient modulo this congruence was found to be isomorphic to the weak order on a Coxeter group of type $A_3$.
The characterization of congruences in terms of shards allows us to revisit this example from a geometric point of view.
Contracting the two chosen edges corresponds to removing the two shards indicated with arrows in Figure~\ref{B3shards}.
Removing these two shards forces the removal of all shards below them in the shard poset.
Figure~\ref{B3shardsA3} depicts the shards whose removal is \emph{not} forced.
(Gaps between intersecting shards have been closed in this illustration.)
\begin{figure}
\scalebox{0.98}{\includegraphics{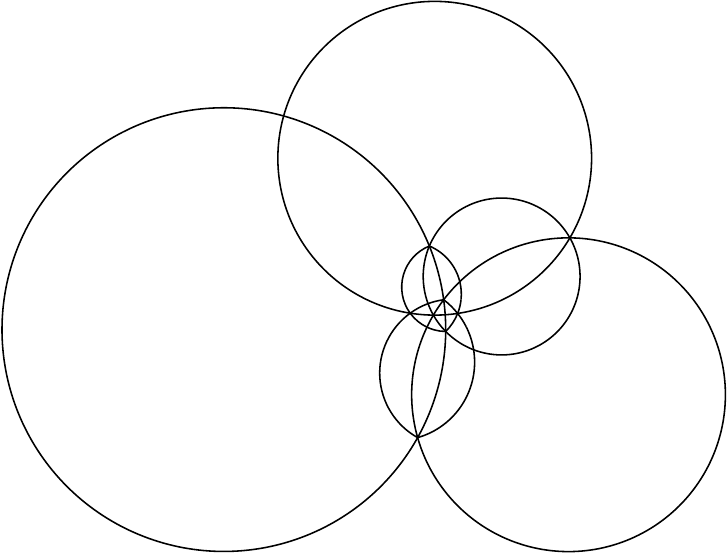}}
\caption{Removing shards from a Coxeter group of type $B_3$}
\label{B3shardsA3}
\end{figure}
The resulting fan is piecewise-linearly (but not linearly) equivalent to the fan defined by a Coxeter arrangement of~type~$A_3$.
\end{example}

Let $\alpha$ denote the involution $w\mapsto ww_0$ on $W.$
This is an anti-automorphism of the weak order.
For each congruence $\Theta$ on $W,$ let $\alpha(\Theta)$ be the \newword{antipodal congruence} to $\Theta$, defined by 
$x\equiv y\mod\alpha(\Theta)$ if and only if $\alpha(x)\equiv \alpha(y)\mod\Theta$.
The involution $\alpha$ induces an anti-isomorphism from $W/\Theta$ to $W/(\alpha(\Theta))$.

The following is a restatement of part of ~\cite[Proposition 6.13]{congruence}.
\begin{proposition}
\label{dual cong}
Let $\Theta$ be a lattice congruence on $W$, let $r,s\in S$, and suppose $k\in{\set{2,3,\ldots,m(r,s)-1}}$.
Then $\alt_k(r,s)$ is contracted by $\Theta$ if and only if $\alt_{k'}(s,r)$ is contracted by $\alpha(\Theta)$, where $k'=m(r,s)-k+1$.
\end{proposition}

The following lemma is part of \cite[Lemma~3.11]{shardint}, rephrased in the special case of the weak order.
(Cf. \cite[Lemma~3.9]{congruence}.)
\begin{lemma} \label{whole}
A shard $\Sigma$ is an entire hyperplane if and only if it is $H_s$ for some $s\in S$.
\end{lemma}

The following lemma is a rephrasing of \cite[Lemma~3.12]{shardint} in the special case of the weak order.
(Cf.\ \cite[Lemma~4.6]{sort_camb}.) 
It implies in particular that a shard of degree 2 is only arrowed by shards of degree~1.
\begin{lemma}\label{half}
Let $\Sigma$ be a shard contained in a reflecting hyperplane $H_t$.
The $\Sigma$ has exactly one facet if and only if $t\not\in S$ but $t\in W_{\set{r,s}}$ for some distinct $r,s\in S$.
In that case, the unique facet of $\Sigma$ is $H_r\cap H_s$.
\end{lemma}

We now use the machinery of shards to prove Theorem~\ref{edge cong}.

\begin{proof}[Proof of Theorem~\ref{edge cong}]
Recall that the first assertion of the theorem has already been established.
As before, let $\Theta_E$ be the lattice congruence whose classes are the fibers of $\eta_E$.
It remains only to prove that $\Theta_E$ is generated by the join-irreducible elements $\alt_k(r,s)$ and $\alt_k(s,r)$ for all $\set{r,s}\in E$ and $k=2,3,\ldots,m(r,s)-1$.

Let $s\in S$ and $w\in W$ have $ws\covered w$ and let $t\in T$ be the reflection $wsw^{-1}$, so that $ws=tw$. 
The inversion set of $w_{J_i}$ is $\inv(w)\cap W_{J_i}$ and the inversion set of an element determines the element, so $(ws)_{J_i}=w_{J_i}$ if and only if $t\not\in W_{J_i}$.
Thus the edge $ws\covered w$ is contracted by $\Theta_E$ if and only if $wsw^{-1}\not\in W_{J_i}$ for all $i\in\set{1,\ldots,k}$.
Equivalently, a shard is removed by $\Theta_E$ if and only if the reflecting hyperplane containing it is $H_t$ for some reflection $t$ with $t\not\in W_{J_i}$ for all $i\in\set{1,\ldots,k}$.

Let $\Sigma_k(r,s)$ be the shard associated to the join-irreducible element $\alt_k(r,s)$ for each $k\in\set{2,\ldots,m(r,s)-1}$.
To prove the theorem, we will show that every shard removed by $\Theta_E$ is forced by the removal of all shards of the form $\Sigma_k(r,s)$ and $\Sigma_k(s,r)$ for each edge $r$---$s$ in $E$.
Specifically, let $\Sigma$ be a shard removed by $\Theta_E$.
We complete the proof by establishing the following claim:
If $\Sigma$ is not $\Sigma_k(r,s)$ or $\Sigma_k(s,r)$ for some edge $r$---$s$ in $E$ and some $k\in\set{2,\ldots,m(r,s)-1}$, then there exists another shard $\Sigma'$, also removed by $\Theta_E$, such that $\Sigma'\to\Sigma$ in the shard digraph.
Since the shard digraph is acyclic, the claim will imply that, for every shard $\Sigma$ removed by $\Theta_E$, there is a directed path from a shard $\Sigma_k(r,s)$ or $\Sigma_k(s,r)$ to $\Sigma$.

We now prove the claim.
Since $\Sigma$ is removed by $\Theta_E$, it is contained in a reflecting hyperplane $H_t$ with $t\not\in W_{J_i}$ for all $i\in\set{1,\ldots,k}$.
Consider a reflecting hyperplane $H_{t'}$ that cuts $H_t$ to define a facet of $\Sigma$.
Then $H_{t'}$ is basic in the rank-two subarrangement $\A'$ containing $H_{t'}$ and $H_t$, while $H_t$ is not basic in $\A'$.
The other basic hyperplane, $H_{t''}$, of $\A'$ also cuts $H_t$ to define the same facet of $\Sigma$.
Thus there exists a shard $\Sigma'$ contained in $H_{t'}$ such that $\Sigma'\to\Sigma$, and there exists a shard $\Sigma''$ contained in $H_{t''}$ such that $\Sigma''\to\Sigma$.

If $t'\not\in W_{J_i}$ for all $i\in\set{1,\ldots,k}$ then $\Sigma'$ is removed by $\Theta$.
Similarly, if $t''\not\in W_{J_i}$ for all $i\in\set{1,\ldots,k}$ then $\Sigma''$ is removed by $\Theta$.
Otherwise, there exists $i\in\set{1,\ldots,k}$ with $t'\in W_{J_i}$ and $j\in\set{1,\ldots,k}$ with $t''\in W_{J_j}$.
A reflection is in a standard parabolic subgroup $W_J$ if and only if its reflecting hyperplane contains the intersection $\bigcap_{s\in J}H_s$.
In particular, $t$ is in $W_{J_i\cup J_j}$.
Furthermore $i\neq j$, because if $i=j$, then $t$ is in $W_{J_i}$, contradicting the fact that $\Sigma$ is removed by $\Theta_E$.

The positive root $\beta_t$ is in the positive linear span of $\beta_{t'}$ and $\beta_{t''}$.
But the simple root coordinates (coordinates with respect to the basis $\Pi$) of $\beta_{t'}$ and $\beta_{t''}$ are supported on the disjoint subsets $\set{\alpha_s:s\in J_i}$ and $\set{\alpha_s:s\in J_j}$.
Thus the simple root coordinates of $\beta_{t'}$ are determined, up to scaling, by the simple root coordinates of $\beta_t$, by restricting the coordinates of $\beta_t$ to simple roots in $\set{\alpha_s:s\in J_i}$.
The simple root coordinates of $\beta_{t''}$ are determined similarly, up to scaling, by the simple root coordinates of $\beta_t$.

We conclude that there is at most one facet of $\Sigma$ defined by hyperplanes $H_{t'}$ and $H_{t''}$ such that there exists $i\in\set{1,\ldots,k}$ with $t'\in W_{J_i}$ and $j\in\set{1,\ldots,k}$ with $t''\in W_{J_j}$.
Thus if $\Sigma$ has more than one facet, then we can use one of these facets to find, as above, a shard $\Sigma'$, removed by $\Theta_E$, with $\Sigma'\to\Sigma$.
If $\Sigma$ has only one facet, then Lemma~\ref{half} says that the facet of $\Sigma$ is defined by $H_r$ and $H_s$ for some $r,s\in S$.
In this case, the join-irreducible element $j(\Sigma)$ associated to $\Sigma$ is in $W_{\set{r,s}}$ but not in $\set{r,s}$.
Thus $j(\Sigma)$ is $\alt_k(r,s)$ or $\alt_k(s,r)$ for some $k\in\set{2,3,\ldots,m(r,s)-1}$.
Equivalently $\Sigma=\Sigma_k(r,s)$ or $\Sigma=\Sigma_k(s,r)$ for some $k\in\set{2,3,\ldots,m(r,s)-1}$.
Since $\Sigma$ is removed by $\Theta_E$, the edge $r$---$s$ is in $E$.
Since $j(\Sigma)\not\in\set{r,s}$, in particular $t\not\in S$, so Lemma~\ref{whole} rules out the possibility that $\Sigma$ has no facets.
We have proved the claim, and thus completed the proof of the theorem.
\end{proof}

We conclude the section by describing the shard digraph for types $A_n$ and $B_n$, quoting results of \cite[Sections~7--8]{congruence}.
The descriptions use the combinatorial realizations explained in Examples~\ref{An para} and~\ref{Bn para}.

For any subset $A$ of $\set{1,\ldots,n+1}$, let $A^c=[n+1]\setminus A$, let $m$ be the minimum element of $A$ and let $M$ be the maximum element of $A^c$.
Join-irreducible elements of $S_{n+1}$ correspond to nonempty subsets of $[n+1]$ such that $M>m$.
Given such a subset, the corresponding join-irreducible permutation $\gamma$ is constructed by listing the elements of $A^c$ in ascending order followed by the elements of $A$ in ascending order.
The poset $\Irr(\Con(S_{n+1}))$ can be described combinatorially, as a partial order on join-irreducible elements, in terms of these subsets.
(See \cite[Theorem~8.1]{congruence}.)
In this paper, we do not need the details, except for the case $n=3$, which is shown in Figure~\ref{IrrConA3 fig}.
\begin{figure}
\centerline{\scalebox{0.9}{\includegraphics{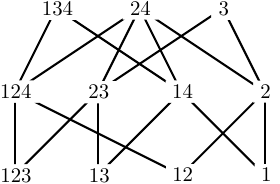}}}
\caption{$\Irr(\Con(S_4))$.}
\label{IrrConA3 fig}
\end{figure}

We do, however, need details of the description of $\Irr(\Con(B_n))$.
A \newword{signed subset} $A$ is a subset of $\set{\pm1,\pm2,\ldots,\pm n}$ such that $A$ contains no pairs $\set{-i,i}$.
Given a nonempty signed subset $A$, let $m$ be the minimum element of $A$.
If $|A|=n$, let $M$ be $-m$ and otherwise, let $M$ be the maximum element of $\set{1,\ldots,n}\setminus \set{|a|:a\in A}$.
The join-irreducible elements of a Coxeter group of type $B_n$ are in bijection with the nonempty signed subsets of $\set{1,\ldots,n}$ with $M>m$.
Given such a signed subset $A$, the corresponding join-irreducible permutation $\gamma$ has one-line notation given by the elements of $\set{1,\ldots,n}\setminus \set{|a|:a\in A}$ in ascending order, followed by the elements of $A$ in ascending order.
The reflection $t$ associated to the unique cover relation down from $\gamma$ is $(M\,\,\,m)(-m\,\,\,-M)$ if $m\neq-M$ or $(m\,\,\,M)$ if $m=-M$.
The associated shard is in the hyperplane $H_t$.

Figure~\ref{IrrConB3 fig} shows the poset $\Irr(\Con(B_3))$.
To describe this poset for general $n$, we quote the combinatorial description of the shard digraph in type $B_n$.
Arrows between shards depend on certain combinations of conditions.
The meaning of the letters q, f and r in the labels for the conditions is explained in \cite[Section~7]{congruence}.
We begin with conditions (q1) through~(q6).\\

\begin{figure}
\centerline{\scalebox{0.9}{\includegraphics{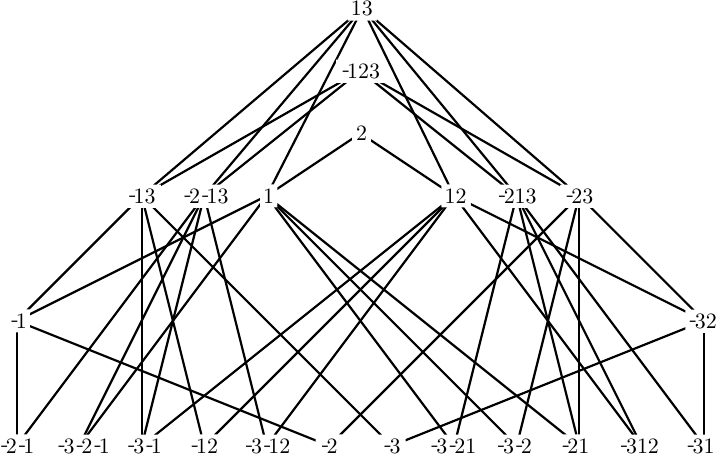}}}
\caption{$\Irr(\Con(B_3))$.}
\label{IrrConB3 fig}
\end{figure}

\begin{tabular}{ll}
(q1)&$-m_1=M_1<M_2=-m_2$.\\
(q2)&$-m_2=M_2=M_1>m_1>0$.\\
(q3)&$M_2=M_1>m_1>m_2\neq -M_2$.\\
(q4)&$M_2>M_1>m_1=m_2\neq -M_2$.\\
(q5)&$-m_2=M_1>m_1>-M_2\neq m_2$.\\
(q6)&$-m_2>M_1>m_1=-M_2\neq m_2$.
\end{tabular}\\

Next we define condition (f), which depends on a parameter in $\set{\pm1,\ldots,\pm n}$.
In the following conditions, the superscript ``$c$'' means complementation in the set $\set{\pm1,\pm2,\ldots,\pm n}$ and the notation $(x,y)$ means the open interval between $x$ and $y$. 
For $a\in\set{\pm1,\ldots,\pm n}$, say $A_2$ satisfies condition (f\,:\,$a$) if one of the following holds:\\

\begin{tabular}{ll}
(f1\,:\,$a$)&$a\in A_2$.\\
(f2\,:\,$a$)&$a\in A_2^c\setminus\set{-M_2,-m_2}$ and $-a\not\in A_2\cap(m_2,M_2)$.\\
(f3\,:\,$a$)&$a\in\set{-M_2,-m_2}$ and $(A_2\cup-A_2)^c\cap(m_2,M_2)\cap(-M_2,-m_2)=\emptyset$.
\end{tabular}\\

\noindent
Every $a\in\set{\pm1,\ldots,\pm n}$ satisfies exactly one of the conditions $a\in A_2$, $a\in A_2^c\setminus\set{-M_2,-m_2}$ and $a\in\set{-M_2,-m_2}$ appearing in condition (f).

Finally, conditions (r1) and (r2):\\

\begin{tabular}{ll}
(r1)&$A_1\cap(m_1,M_1)=A_2\cap(m_1,M_1)$.\\
(r2)&$A_1\cap(m_1,M_1)=-A_2^c\cap(m_1,M_1)$.
\end{tabular}\\

\begin{theorem}
\label{B shard}
$\Sigma_1\to\Sigma_2$ if and only if one of the following combinations of conditions holds:
\begin{enumerate}
\item[1. ] \textup{(q1)} and \textup{(r1)}.
\item[2. ] \textup{(q2)} and \textup{(r1)}.
\item[3. ] \textup{(q3)}, \textup{(f\,:\,$m_1$)} and \textup{(r1)}.
\item[4. ] \textup{(q4)}, \textup{(f\,:\,$M_1$)} and \textup{(r1)}.
\item[5. ] \textup{(q5)}, \textup{(f\,:\,$-m_1$)} and \textup{(r2)}.
\item[6. ] \textup{(q6)}, \textup{(f\,:\,$-M_1$)} and \textup{(r2)}.
\end{enumerate}
\end{theorem}

\section{Decreasing edge labels}\label{dihedral sec}
Theorem~\ref{edge factor} reduces the problem of classifying compressive homomorphisms $\eta$ between weak orders to the special case of compressive homomorphisms that do not erase any edges.
Equivalently, these are the compressive homomorphisms between finite Coxeter groups that restrict to (unlabeled) graph isomorphisms of diagrams.
As a byproduct, Theorem~\ref{edge factor} allows us to further reduce the problem to the case where the diagrams are connected (or equivalently, where the Coxeter groups are irreducible).
We continue to take $S$ to be the set of simple generators of $W$ and of $W'$ and assume that $\eta:W\to W'$ is such a homomorphism that fixes $S$ pointwise, so that the diagrams of $W$ and $W'$ are identical as unlabeled graphs.
Recall that Proposition~\ref{diagram facts} says that for each edge in the diagram of $W$, the corresponding edge in the diagram of $W'$ has weakly smaller label.

Up to now, the classification of surjective lattice homomorphisms between weak orders has proceeded by uniform arguments, rather than arguments that are specific to particular families in the classification of finite Coxeter groups.
To complete the classification of homomorphisms, we now turn to the classification of finite Coxeter groups, which tells us in particular that there are very few cases remaining to consider.
We need to study the cases where $(W,W')$ are $(I_2(m),I_2(m'))$ for $m'<m$ or $(B_n,A_n)$, $(F_4,A_4)$, $(H_3,A_3)$, $(H_3,B_3)$, $(H_4,A_4)$, or $(H_4,B_4)$. 

The cases where $(W,W')$ are $(I_2(m),I_2(m'))$ for $m'<m$ are easily understood by considering Example~\ref{i25 shards}, where $W$ is $I_2(5)$.
It is apparent from Figure~\ref{i25} that a congruence $\Theta$ has the property that $W/\Theta$ is isomorphic to the weak order on $I_2(4)$ if and only if $\Theta$ contracts exactly one of the join-irreducible elements $rs$, $rsr$, and $rsrs$ and exactly one of the join-irreducible elements $sr$, $srs$, $srsr$.

More generally, let $W$ be a Coxeter group of type $I_2(m)$ with $S=\set{r,s}$.
A congruence $\Theta$ on $W$ has the property that $W/\Theta$ is isomorphic to the weak order on $I_2(m')$ if and only if $\Theta$ contracts exactly $m-m'$ join-irreducible elements of the form $\alt_k(r,s)$ for $k\in\set{2,\ldots,m-1}$ and exactly $m-m'$ join-irreducible elements of the form $\alt_k(s,r)$ for $k\in\set{2,\ldots,m-1}$.
This result for dihedral groups is simple, but its importance for the general case is underscored by Theorem~\ref{diagram uniqueness}.

In Sections~\ref{Bn Sn+1 sec} and~\ref{exceptional sec}, we consider the remaining possibilities for the pair $(W,W')$ and, in particular, complete the proofs of Theorems~\ref{existence} and~\ref{diagram uniqueness}.  
As mentioned earlier, Theorem~\ref{main} follows from Theorems~\ref{para factor} and~\ref{existence}.

\section{Homomorphisms from $B_n$ to $S_{n+1}$}\label{Bn Sn+1 sec}
We realize the groups $S_{n+1}$ and $B_n$ as in Examples~\ref{An para} and~\ref{Bn para}.
A surjective homomorphism $\eta$ from $B_n$ to $S_{n+1}$ restricts, by Proposition~\ref{diagram facts}, to a surjective homomorphism from $(B_n)_{\set{s_0,s_1}}\cong B_2$ to $(S_{n+1})_{\set{s_1,s_2}}\cong S_3$.
Thus, as discussed in Section~\ref{dihedral sec}, the congruence associated to $\eta$ contracts exactly one of the join-irreducible elements $s_0s_1,s_0s_1s_0$, and exactly one of the join-irreducible elements $s_1s_0,s_1s_0s_1$.
We will see that each of the four choices leads to a unique surjective homomorphism and that three of the four choices are associated to homogeneous congruences of degree $2$.

\subsection{Simion's homomorphism}\label{simion sec}
We consider first the case where $s_0s_1$ and $s_1s_0s_1$ are contracted, but we start not by contracting join-irreducibles, but by giving a map from $B_n$ to $S_{n+1}$.
This map was first defined by Simion~\cite{Simion} in connection with a construction of the type-B associahedron (also known as the cyclohedron).

Let $\eta_\sigma:B_n\to S_{n+1}$ map $\pi\in B_n$ to the permutation constructed as follows:
Construct the sequence $(-\pi_n)(-\pi_{n-1})\cdots(-\pi_1)\,0\,\pi_1\cdots\pi_{n-1}\pi_n$, extract the subsequence consisting of nonnegative entries and add $1$ to each entry.
Thus for example, for $\pi=3(-4)65(-7)(-1)2\in B_7$, we construct the sequence 
\[(-2)17(-5)(-6)4(-3)03(-4)65(-7)(-1)2\] 
and extract the subsequence $17403652$, so that $\eta_\sigma(\pi)$ is $28514763\in S_8$.
We will prove the following theorem:

\begin{theorem}\label{simion thm}
The map $\eta_\sigma$ is a surjective lattice homomorphism from $B_n$ to $S_{n+1}$.
Its fibers constitute the congruence generated by $s_0s_1$ and $s_1s_0s_1$.
Furthermore, $\eta_\sigma$ is the unique surjective lattice homomorphism from $B_n$ to $S_{n+1}$ whose restriction to $(B_n)_{\set{s_0,s_1}}$ agrees with $\eta_\sigma$.
\end{theorem}

Suppose $\pi\covered \tau$ in the weak order on $B_n$.
Then the one-line notations for $\pi$ and $\tau$ differ in one of two ways:
Either they agree except in the sign of the first entry or they agree except that two adjacent entries of $\pi$ are transposed in $\tau$.
If they agree except in the sign of the first entry, then $\eta_\sigma(\pi)\neq\eta_\sigma(\tau)$.
If they agree except that two adjacent entries of $\pi$ are transposed in $\tau$, then $\eta_\sigma(\pi)=\eta_\sigma(\tau)$ if and only if the two adjacent entries have opposite signs.
We have proved the following:
\begin{proposition}\label{simion cover}
If $\pi\covered \tau$ in the weak order on $B_n$, then $\eta_\sigma(\pi)=\eta_\sigma(\tau)$ if and only if the reflection $t$ associated to the cover $\pi\covered\tau$ is $(i\,\,-j)(j\,\,-i)$ for some $i$ and $j$ with $1\le i<j\le n$. 
\end{proposition}

We give a similar characterization of the congruence generated by $s_0s_1$ and~$s_1s_0s_1$:
\begin{proposition}\label{simion finest}
The congruence generated by $s_0s_1$ and $s_1s_0s_1$ removes a shard $\Sigma$ if and only if $\Sigma$ is contained in a hyperplane $H_t$ such that $t=(i\,\,-j)(j\,\,-i)$ for some $i$ and $j$ with $1\le i<j\le n$.
\end{proposition}
\begin{proof}
Let $C$ be the set of signed subsets corresponding to shards contained in hyperplanes $H_t$ such that $t=(i\,\,-j)(j\,\,-i)$ for some $i$ and $j$ with $1\le i<j\le n$.
Thus $C$ is the set of all nonempty signed subsets $A$ with $m<0$ and $m\neq-M$.
(Equivalently, $A\not\subseteq \set{1,\ldots,n}$ and $|A|<n$.)
The assertion of the proposition is that $C$ is the set of signed subsets representing shards removed by the congruence generated by $s_0s_1$ and $s_1s_0s_1$.
The join-irreducible element $s_0s_1$ has one-line notation $2(-1)34\cdots n$, and the join-irreducible element $s_1s_0s_1$ is $1(-2)34\cdots n$.
The corresponding signed subsets are $\set{-1,3,4,\ldots, n}$ and $\set{-2,3,4,\ldots, n}$, both of which are in $C$.

To establish the ``only if'' assertion of the proposition, we show that no set $A_1\in C$ arrows a set $A_2\not\in C$.
Indeed, if $A_1\in C$ and $A_2\not\in C$, then $m_1<0$ and $m_1\neq-M_1$ and either $m_2>0$ or $m_2=-M_2$.
The fact that $m_1<0$ and $m_1\neq-M_1$ rules out conditions (q1) and (q2).
Whether $m_2>0$ or $m_2=-M_2$, conditions (q3)--(q6) are easily ruled out as well.

To establish the ``if'' assertion, we show that every set $A_2\in C$, except the sets $\set{-1,3,4,\ldots, n}$ and $\set{-2,3,4,\ldots, n}$, is arrowed to by another set ${A_1\in C}$.
Since the shard digraph is acyclic, this implies that $\set{-1,3,4,\ldots, n}$ and $\set{-2,3,4,\ldots, n}$ are the unique sources in the restriction of the digraph to $C$.
Let $A_2\in C$, so that $m_2<0$ and $m_2\neq-M_2$, with $A_2$ not equal to $\set{-1,3,4,\ldots, n}$ or $\set{-2,3,4,\ldots, n}$.
We will produce a signed subset $A_1\in C$ with $A_1\to A_2$ by considering several cases.

\noindent
\textbf{Case 1:} $|A_2|<n-1$.
Then let $A_1=A_2\cup\set{M_2}$.
Then $M_2>M_1>0>m_1=m_2\neq-M_2$, so (q4) holds.
Furthermore, $|A_1|<n$, so $m_1\neq M_1$, and thus $A_1\in C$.
We have $M_1\in A_2^c\setminus\set{-M_2,-m_2}$.
By definition of $M_1$, the element $-M_1$ is not in $A_1$, and since $A_2\subset A_1$, we conclude that $-M_1\not\in A_2$.
Thus (f2\,:\,$M_1$) holds.
Also, (r1) holds, so $A_1\to A_2$ in the shard digraph.

\noindent
\textbf{Case 2:} $|A_2|=n-1$.
Since $A_2\in C$, this is the only alternative to Case 1.
By hypothesis, we have ruled out the possibilities $(m,M)=(-1,2)$ and $(m,M)=(-2,1)$.

\noindent
\textbf{Subcase 2a:} $A_2\cap[1,M_2-1]\neq\emptyset$.
(The notation $[a,b]$ stands for the closed interval between $a$ and $b$.)
Define $M_1$ to be the maximum element of $A_2\cap[1,M_2-1]$ and define $A_1$ to be $(A_2\cup\set{M_2})\setminus\set{M_1}$.
Thus $M_1$ is indeed the maximum element of $\set{1,\ldots,n}\setminus \set{|a|:a\in A_1}$.
Conditions (q4), (f1\,:\,$M_1$), and (r1) hold.
We have $m_1=m_2<0$ and $|A_1|=|A_2|<n$, to $A_1$ is in $C$.

\noindent
\textbf{Subcase 2b:} $A_2\cap[1,M_2-1]=\emptyset$ and $m_2>-M_2$.
By definition of $M_2$, neither $M_2$ nor $-M_2$ is in $A_2$.
Thus since $|A_2|=n-1$, every other element $i$ of $\set{1,\ldots,n}$ has either $i$ or $-i$ in $A_2$.
Since $A_2\cap[1,M_2-1]=\emptyset$, we conclude that the interval $[-M_2+1,-1]$ is contained in $A_2$.
Now the fact that $m_2>-M_2$ implies that $m_2=-M_2+1$.
If $m_2=-1$ then $A_2=\set{-1,3,4,\ldots, n}$, a possibility that is ruled out by hypothesis.
Define $A_1$ to be $A_2\setminus\set{m_2}$.
Then $m_1=m_2+1<0$ and $M_1=M_2$.
Also $|A_1|=n-2$, so $A_1$ is in $C$.
Conditions (q3), (f1\,:\,$m_1$) and (r1) hold.

\noindent
\textbf{Subcase 2c:} $A_2\cap[1,M_2-1]=\emptyset$, $m_2<-M_2$, and $A_2\cap[m_2+1,-1]\neq\emptyset$.
Let $A_1=(A_2\cup\set{-m_2})\setminus\set{m_2}$.
Since $A_2\cap[m_2+1,-1]\neq\emptyset$, the minimum element $m_1$ of $A_1$ is negative, and since in addition $|A_1|=|A_2|<n$, the set $A_1$ is in $C$.
Conditions (q3), (f1\,:\,$m_1$), and (r1) hold, with condition (r1) using the fact that $-m_2>M_1$.

\noindent
\textbf{Subcase 2d:} $A_2\cap[1,M_2-1]=\emptyset$, $m_2<-M_2$, and $A_2\cap[m_2+1,-1]=\emptyset$.
As in Subcase 2b, besides $M_2$, every element $i$ of $\set{1,\ldots,n}$ has either $i$ or $-i$ in $A_2$.
Thus the fact that $A_2\cap[1,M_2-1]=\emptyset$ and $A_2\cap[m_2+1,-1]=\emptyset$ implies that $[1,-m_2-1]\cap[1,M_2-1]=\emptyset$.
In other words, either $m_2=-1$ or $M_2=1$, but the fact that $m_2<-M_2$ rules out the possibility that $m_2=-1$.
Thus $M_2=1$.
Since $A_2\cap[m_2+1,-1]=\emptyset$, the set $A_2$ equals $\set{m_2}\cup(\set{1,\ldots,n}\setminus\set{1,-m_2})$.
If $m_2=-2$, then $A_2=\set{-2,3,4,\ldots, n}$, which was also ruled out by hypothesis.
The set $A_1=(A_2\setminus\set{m_2,-m_2-1})\cup\set{-m_2,m_2+1}$ is in $C$.
We have $m_1=m_2+1<-1$ and $M_1=M_2=1$, so conditions (q3), (f2\,:\,$m_1$), and (r1) hold.
\end{proof}

%
So far, we know nothing about how $\eta_\sigma$ relates to the weak order on $S_{n+1}$.
But Propositions~\ref{simion cover} and~\ref{simion finest} constitute a proof of the second assertion of Theorem~\ref{simion thm}:  that the congruence $\Theta_\sigma$ defined by the fibers of $\eta_\sigma$ is generated by $s_0s_1$ and $s_1s_0s_1$.

Proposition~\ref{simion cover} implies that the bottom elements of $\Theta_\sigma$ are exactly those signed permutations whose one-line notation consists of a (possibly empty) sequence of negative entries followed by a (possibly empty) sequence of positive entries.
The restriction of $\eta_\sigma$ to bottom elements is a bijection to $S_{n+1}$, with inverse map described as follows:
If $\tau\in S_{n+1}$ and $\tau_i=1$, then the inverse map takes $\tau$ to the signed permutation whose one-line notation is 
\[(-\tau_{i-1}+1)(-\tau_{i-2}+1)\cdots(-\tau_1+1)(\tau_{i+1}-1)(\tau_{i+2}-1)\cdots(\tau_{n+1}-1).\]
This inverse map is easily seen to be order-preserving.
The map $\eta_\sigma$ is also easily seen to be order-preserving, so its restriction is as well.
We have shown that the restriction of $\eta_\sigma$ to bottom elements of $\Theta_\sigma$ is an isomorphism to the weak order on~$S_{n+1}$.

The natural map from $B_n$ to $B_n/\Theta_\sigma$ is a surjective homomorphism, and $B_n/\Theta_\sigma$ is isomorphic to the subposet of $B_n$ induced by bottom elements of $\Theta_\sigma$.
Since $\eta_\sigma$ equals the composition of the natural map, followed by the isomorphism to the poset of bottom elements, followed by the isomorphism to $S_{n+1}$, it is a surjective homomorphism.
Now, if $\eta$ is any surjective homomorphism agreeing with $\eta_\sigma$ on $(B_n)_{\set{s_0,s_1}}$, the associated congruence $\Theta$ contracts $s_0s_1$ and $s_1s_0s_1$.
Thus $\Theta$ is a coarsening of the congruence $\Theta_\sigma$ associated to $\eta_\sigma$.
But then since both $B_n/\Theta$ and $B_n/(\Theta_\sigma)$ are isomorphic to $S_{n+1}$, they must coincide.
We have completed the proof of Theorem~\ref{simion thm}.

\subsection{A non-homogeneous homomorphism}\label{nonhom sec}
In this section, we consider surjective homomorphisms whose congruence contracts $s_0s_1s_0$ and $s_1s_0$.
We begin by defining a homomorphism $\eta_\nu$ that is combinatorially similar to Simion's homomorphism.
We will see that, whereas $\eta_\sigma$ is homogeneous of degree $2$, $\eta_\nu$ is non-homogeneous.
However, $\eta_\nu$ is still of low degree:  it is generated by contracting join-irreducible elements of degrees $2$ and $3$.

Let $\eta_\nu:B_n\to S_{n+1}$ map $\pi\in B_n$ to the permutation by extracting the subsequence of $(-\pi_n)(-\pi_{n-1})\cdots(-\pi_1)\pi_1\cdots\pi_{n-1}\pi_n$ consisting of values greater than or equal to $-1$, changing $-1$ to $0$, then adding $1$ to each entry.
Thus for example, for $\pi=3(-4)65(-7)(-1)2\in B_7$, we extract the subsequence $174365(-1)2$, so that $\eta_\nu(\pi)$ is $28547613$.
We will prove the following theorem:

\begin{theorem}\label{nonhom thm}
The map $\eta_\nu$ is a surjective lattice homomorphism from $B_n$ to $S_{n+1}$.
Its fibers constitute the congruence generated by $s_0s_1s_0$, $s_1s_0$, $s_1s_0s_1s_2$, and $s_2s_1s_0s_1s_2$.
Furthermore, $\eta_\nu$ is the unique surjective lattice homomorphism from $B_n$ to $S_{n+1}$ whose restriction to $(B_n)_{\set{s_0,s_1}}$ agrees with $\eta_\nu$.
\end{theorem}

The outline of the proof of Theorem~\ref{nonhom thm} is the same as the proof of Theorem~\ref{simion thm}.
We first determine when signed permutations $\pi\covered \tau$ in $B_n$ map to the same element of $S_{n+1}$.
If $\pi$ and $\tau$ agree except in the sign of the first entry (with $\pi_1$ necessarily positive), then $\eta_\nu(\pi)=\eta_\nu(\tau)$ if and only if $\pi_1>1$.
If they agree except that two adjacent entries of $\pi$ are transposed in $\tau$, then $\eta_\nu(\pi)=\eta_\nu(\tau)$ if and only if the two transposed entries have opposite signs and neither is $\pm1$.
Thus:
\begin{proposition}\label{nonhom cover}
If $\pi\covered \tau$ in the weak order on $B_n$, then $\eta_\nu(\pi)=\eta_\nu(\tau)$ if and only if the reflection $t$ associated to the cover $\pi\covered\tau$ is either $(i\,\,-i)$ for some $i>1$ or $(i\,\,-j)(j\,\,-i)$ for some $i$ and $j$ with $2\le i<j\le n$.
\end{proposition}

We now show that the congruence generated by $s_0s_1s_0$, $s_1s_0$, $s_1s_0s_1s_2$, and $s_2s_1s_0s_1s_2$ agrees with the fibers of $\eta_\nu$, as described in Proposition~\ref{nonhom cover}.

\begin{proposition}\label{nonhom finest}
In the congruence generated by $s_0s_1s_0$, $s_1s_0$, $s_1s_0s_1s_2$, and $s_2s_1s_0s_1s_2$ a shard is removed if and only if it is contained in a hyperplane $H_t$ with $t=(i\,\,-i)$ and $i>1$ or $t=(i\,\,-j)(j\,\,-i)$ and $2\le i<j\le n$.
\end{proposition}
\begin{proof}
Let $C$ be the set of signed subsets corresponding to shards contained in hyperplanes $H_t$ such that $t$ is either $(i\,\,-i)$ for some $i>1$ or $(i\,\,-j)(j\,\,-i)$ for some $i$ and $j$ with $2\le i<j\le n$.
Thus $C$ is the set of all nonempty signed subsets $A$ with $m<-1$ and $M>1$.
The elements $s_0s_1s_0$, $s_1s_0$, $s_1s_0s_1s_2$, and $s_2s_1s_0s_1s_2$ correspond respectively to the signed sets $\set{-2,-1,3,4,\ldots,n}$, $\set{-2,1,3,4,\ldots,n}$, $\set{-2,4,5,\ldots,n}$, and $\set{-3,4,5,\ldots,n}$. 

We first show that no set $A_1\in C$ arrows a set $A_2\not\in C$.
Let $A_1\in C$ and $A_2\not\in C$.
Then $m_1<-1$ and $M_1>1$, while either $m_2\ge-1$ or $M_2=1$.
The fact that $m_1<-1$ rules out condition (q2).
If $M_2=1$, then the fact that $M_1>1$ rules out conditions (q1), (q3) and (q4), and the fact that $m_1<-1$ rules out (q5) and (q6).
Suppose $m_2\ge -1$.
Then the equality $M_2=-m_2$ in (q1) fails unless $m_2=-1$, but in that case $M_2=1$, so (q1) is ruled out.
The fact that $m_1<-1$ rules out (q3) and (q4).
Since $-m_2\le 1$ and $M_1>1$, conditions (q5) and (q6) fail as well.

Next, let $A_2\in C$, so that $m_2<-1$ and $M_2>1$, with $A_2$ not being one of the sets $\set{-2,-1,3,4,\ldots,n}$, $\set{-2,1,3,4,\ldots,n}$, $\set{-2,4,5,\ldots,n}$, or $\set{-3,4,5,\ldots,n}$.
We will produce a signed subset $A_1\in C$ with $A_1\to A_2$.

\noindent
\textbf{Case 1:} $m_2=-M_2$.
Then $|A_2|=n$.
If $m_2=-2$, then $A_2=\set{-2,-1,3,4,\ldots,n}$ or $A_2=\set{-2,1,3,4,\ldots,n}$, but these are ruled out, so $m_2<-2$.
Let $A_1=(A_2\cup\set{-m_2,m_2+1})\setminus\set{m_2,-m_2-1}$.
Then $m_1=m_2+1$ and $M_1=M_2-1=-m_1$.
Conditions (q1) and (r1) hold, so $A_1\to A_2$.
Also $-m_1=M_1>1$, so $A_1\in C$.

\noindent
\textbf{Case 2:} $m_2\neq-M_2$.
Then $|A_2|<n$.

\noindent
\textbf{Subcase 2a:} $(A_2\cup -A_2)^c\cap[2,M_2-1]\neq\emptyset$.
Let $A_1=A_2\cup\set{M_2}$.
Then $m_1=m_2$ and $M_1$ is the maximal element of $(A_2\cup -A_2)^c\cap[2,M_2-1]$, so $A_2\in C$.
Furthermore, $M_2>M_1>m_1=m_2\neq-M_2$, so (q4) holds.
Also (f2\,:\,$M_1$) and (r1) hold.

\noindent
\textbf{Subcase 2b:} $(A_2\cup -A_2)^c\cap[2,M_2-1]=\emptyset$ and $A_2\cap[2,M_2-1]\neq\emptyset$.
Let $b$ be the largest element of $A_2\cap[2,M_2-1]\neq\emptyset$ and define $A_1=(A_2\cup\set{M_2})\setminus\set{b}$.
Then $m_1=m_2$ and $M_1=b$, so $A_1\in C$.
Also (q4), (f1\,:\,$M_1$) and (r1) hold.

\noindent
\textbf{Subcase 2c:} $(A_2\cup -A_2)^c\cap[2,M_2-1]=\emptyset$ and $A_2\cap[2,M_2-1]=\emptyset$.
In this case, $A_2$ contains $\set{-M_2+1,-M_2+2,\ldots,-2}$.
If $m_2<-2$ and $M_2>2$ then let $A_1=A_2\setminus\set{m_2}$.
This is in $C$, since $M_1=M_2>1$ and $m_1\le-2$.
Also, (q3), (f1\,:\,$m_1$) and (r1) hold.

If $M_2=2$ then $m_2<-2$.
If also $m_2=-3$ then there are two possibilities, because $A_2=\set{-3,4,5,\ldots,n}$ is ruled out.
If $A_2=\set{-3,1,4,5,\ldots,n}$ then let $A_1=\set{-2,1,3,4,\ldots,n}\in C$.
If $A_2=\set{-3,-1,4,5,\ldots,n}$ then let $A_1=\set{-2,-1,3,4,\ldots,n}\in C$.
In either case, (q6), (f3\,:\,$-M_1$) and (r2) hold.
If $m_2<-3$ then let $A_1=(A_2\cup\set{m_2+1,-m_2})\setminus\set{-m_2-1,m_2}$.
Then $m_1=m_2+1<-2$ and $M_1=2$, so $A_1\in C$.
Conditions (q3) and (r1) hold, along with either (f1\,:\,$m_1$) or (f2\,:\,$m_1$).

If $m_2=-2$ then $M_2>2$.
Since $A_2$ contains $\set{-M_2+1,-M_2+2,\ldots,-2}$, we conclude that $M_2=3$.
Since $A_2=\set{-2,4,5,\ldots,n}$ is ruled out, there are two possibilities.
If $A_2=\set{-2,1,4,5,\ldots,n}$ then let $A_1={-2,1,3,4,5,\ldots,n}$.
If $A_2=\set{-2,-1,4,5,\ldots,n}$ then let $A_1={-2,-1,3,4,5,\ldots,n}$.
In either case, (q4), (f3\,:\,$M_1$), and (r1) hold.
\end{proof}

We have showed that the fibers of $\eta_\nu$ are a lattice congruence on $B_n$ satisfying the second assertion of Theorem~\ref{nonhom thm}.
Let $\Theta_\nu$ be this congruence.
The map $\eta_\nu$ is obviously order-preserving, so just as in the proof of Theorem~\ref{simion thm}, we will show that $\eta_\nu$ restricts to a bijection from bottom elements of $\Theta_\nu$-classes to permutations in $S_{n+1}$ and that the inverse of the restriction is order-preserving.

Proposition~\ref{nonhom cover} implies that the bottom elements of $\Theta_\nu$ are exactly those signed permutations whose one-line notation consists of a sequence of positive elements, followed by $\pm1$, then a sequence of negative elements, and finally a sequence of positive elements. 
(Any of these three sequences may be empty.)
Such a signed permutation $\pi_1\cdots\pi_n$ with $\pi_i=1$ and negative entries $\pi_{i+1}\cdots\pi_j$, maps to the permutation 
\begin{multline*}
(-\pi_j+1)(-\pi_{j-1}+1)\cdots(-\pi_{i+1}+1)\,1\,(\pi_1+1)(\pi_2+1)\\\cdots(\pi_{i-1}+1)\,2\,(\pi_{j+1}+1)(\pi_{j+2}+1)\cdots(\pi_n+1).
\end{multline*}
If $\pi_i=-1$ instead, then the bottom element maps to the same permutation, except with the entries $1$ and $2$ swapped.
This restriction of $\eta_\nu$ is a bijection whose inverse takes a permutation $\tau_1\cdots\tau_{n+1}$ with $\tau_i=1$ and $\tau_j=2$, with $i<j$, to the signed permutation 
\[(\tau_{i+1}-1)\cdots(\tau_{j-1}-1)\,1\,(-\tau_{i-1}+1)\cdots(-\tau_1+1)(\tau_{j+1}-1)\cdots(\tau_{n+1}-1).\]
If $\tau_i=2$ and $\tau_j=1$, with $i<j$, then the inverse map takes $\tau$ to the same signed permutation, except with $-1$ in place of $1$.
This inverse is order-preserving, and we have proved the first two assertions of Theorem~\ref{nonhom thm}.

To prove the third assertion, we temporarily introduce the bulky notation $\eta_\nu^{(n)}:B_n\to S_{n+1}$.
(Until now, we had suppressed the explicit dependence of the map $\eta_\nu$ on $n$, and we will continue to do so after this explanation.)
Arguing as in the proof of Theorem~\ref{simion thm}, we easily see that, for each $n\ge 3$, the map $\eta_\nu^{(n)}$ is the unique surjective homomorphism from $B_n$ to $S_{n+1}$ whose restriction to $(B_n)_{\set{s_0,s_1,s_2}}$ is $\eta_\nu^{(3)}$.
Now let $\eta$ be any surjective homomorphism from $B_n$ to $S_{n+1}$ whose restriction to $(B_n)_{\set{s_0,s_1}}$ is $\eta_\nu^{(2)}$.
Let $\Theta$ be the associated congruence on $B_n$.
Then in particular, the congruence defined by the restriction of $\eta$ to $(B_n)_{\set{s_0,s_1,s_2}}\cong B_3$ contracts $s_0s_1s_0$ and $s_1s_0$, corresponding to signed subsets $\set{-2,-1,3}$ and $\set{-2,1,3}$.
Figure~\ref{IrrConB3 squashed fig} shows the poset of signed subsets corresponding to join-irreducible elements in $(B_n)_{\set{s_0,s_1,s_2}}\cong B_3$ not forced to be contracted by the contraction of $s_0s_1s_0$ and $s_1s_0$.
\begin{figure}
\centerline{\scalebox{.9}{\includegraphics{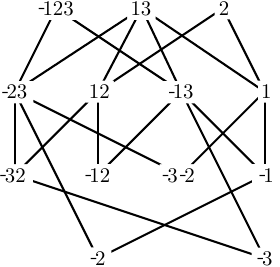}}}
\caption{The complement of an order ideal in $\Irr(\Con(W))$ for $W$ of type $B_3$.}
\label{IrrConB3 squashed fig}
\end{figure}
(This is the complement of an order ideal in the poset of Figure~\ref{IrrConB3 fig}.)

By Proposition~\ref{diagram facts}, the restriction of $\eta$ to $(B_n)_{\set{s_0,s_1,s_2}}$ is a surjective homomorphism to $S_4$.
In particular, the restriction of $\Irr(\Con((B_n)_{\set{s_0,s_1,s_2}}))$ to join-irreducibles not contracted by $\Theta$ is isomorphic to $\Irr(\Con(S_4))$.
Comparing Figure~\ref{IrrConB3 squashed fig} to Figure~\ref{IrrConA3 fig}, we see that $\Theta$ must contract two additional join-irreducible elements, beyond those forced by $s_0s_1s_0$ and $s_1s_0$.
We see, furthermore, that the only two join-irreducible elements that can be contracted, to leave a poset isomorphic to $\Irr(\Con(S_4))$, are those whose signed subsets are $-2$ and $-3$.
We conclude that $\Theta$ contracts $s_1s_0s_1s_2$ and $s_2s_1s_0s_1s_2$.
The second assertion of Theorem~\ref{nonhom thm} says that $\Theta$ is weakly coarser than $\Theta_\nu$.
But since $\eta$ and $\eta_\nu$ are both surjective lattice homomorphisms to $S_{n+1}$, the congruences $\Theta$ and $\Theta_\nu$ have the same number of congruence classes, so $\Theta=\Theta_\nu$.
Thus $\eta$ and $\eta_\nu$ agree up to automorphisms of $S_{n+1}$, but the only nontrivial automorphism of $S_{n+1}$ is the diagram automorphism.
Since both maps take $(B_n)_{\set{s_0,s_1}}$ to $(S_{n+1})_{\set{s_1,s_2}}$, we rule out the diagram automorphism and conclude that $\eta=\eta_\nu$.
Thus completes the proof of Theorem~\ref{nonhom thm}.

\subsection{Two more homogeneous homomorphisms}\label{two more sec}
In this section, we consider the case where $s_0s_1$ and $s_1s_0$ are contracted and the case where $s_0s_1s_0$ and $s_1s_0s_1$ are contracted.
The congruences associated to these cases are dual to each other by Proposition~\ref{dual cong}.

Let $\eta_\delta:B_n\to S_{n+1}$ send $\pi\in B_n$ to the permutation obtained as follows:
If the one-line notation for $\pi$ contains the entry $1$, then construct a sequence 
\[(-\pi_n)(-\pi_{n-1})\cdots(-\pi_1)\,0\,\pi_1\cdots\pi_{n-1}\pi_n,\] 
extract the subsequence consisting of nonnegative entries, and add $1$ to each entry.
If the one-line notation for $\pi$ contains the entry $-1$, then extract the subsequence of $(-\pi_n)(-\pi_{n-1})\cdots(-\pi_1)\pi_1\cdots\pi_{n-1}\pi_n$ consisting of values greater than or equal to $-1$, change $-1$ to $0$, then add $1$ to each entry.
Notice that $\eta_\delta$ is a hybrid of $\eta_\sigma$ and $\eta_\nu$, in the sense that $\eta_\delta(\pi)=\eta_\sigma(\pi)$ if the one-line notation of $\pi$ contains $1$ and $\eta_\delta(\pi)=\eta_\nu(\pi)$ if the one-line notation of $\pi$ contains $-1$.

We will prove the following theorem:

\begin{theorem}\label{delta thm}
The map $\eta_\delta$ is a surjective lattice homomorphism from $B_n$ to $S_{n+1}$.
Its fibers constitute the congruence generated by $s_0s_1s_0$ and $s_1s_0s_1$.
Furthermore, $\eta_\delta$ is the unique surjective lattice homomorphism from $B_n$ to $S_{n+1}$ whose restriction to $(B_n)_{\set{s_0,s_1}}$ agrees with $\eta_\delta$.
\end{theorem}

Theorem~\ref{delta thm} implies in particular that the lattice homomorphism associated to the congruence of Example~\ref{miraculous} is the $n=3$ case of $\eta_\delta$.

Suppose $\pi\covered \tau$ in $B_n$.
First, suppose that $1$ is an entry in the one-line notation of both $\pi$ and $\tau$.
If $\pi$ and $\tau$ agree except in the sign of the first entry, then $\eta_\delta(\pi)\neq\eta_\delta(\tau)$.
If they agree except that two adjacent entries of $\pi$ are transposed in $\tau$, then $\eta_\delta(\pi)=\eta_\delta(\tau)$ if and only if the two adjacent entries have opposite signs.
Next, suppose that $-1$ is an entry in the one-line notation of both $\pi$ and $\tau$.
If $\pi$ and $\tau$ agree except in the sign of the first entry, then since  $-1$ is an entry in the one-line notation of both $\pi$ and $\tau$, we must have $\pi_1>1$, and therefore $\eta_\delta(\pi)=\eta_\delta(\tau)$.
If they agree except that two adjacent entries of $\pi$ are transposed in $\tau$, then $\eta_\nu(\pi)=\eta_\nu(\tau)$ if and only if the two transposed entries have opposite signs and neither is $-1$.
Finally, suppose $\pi$ has the entry $1$ in its one-line notation, but $\tau$ has $-1$.
Then $\pi_1=1$ and $\tau_1=-1$, and $\eta_\delta(\pi)\neq\eta_\delta(\tau)$.
Thus:

\begin{proposition}\label{delta cover}
Suppose $\pi\covered \tau$ in the weak order on $B_n$, and let $t$ be the reflection associated to the cover $\pi\covered\tau$.
Then $\eta_\delta(\pi)=\eta_\delta(\tau)$ if and only if one of the following conditions holds:
\begin{enumerate}
\item[(i)] $t$ is $(i\,\,-j)(j\,\,-i)$ for some $i$ and $j$ with $2\le i<j\le n$.
\item[(ii)] $\pi$ has the entry $1$ in its one-line notation and $t$ is $(1\,\,-j)(j\,\,-1)$ for some $j$ with $1<j\le n$.  
\item[(iii)] $\pi$ has the entry $-1$ in its one-line notation and $t$ is $(i\,\,-i)$ for some $i>1$.
\end{enumerate}
\end{proposition}

As in the previous cases, we now show that the congruence generated by $s_0s_1s_0$ and $s_1s_0s_1$ has a description compatible with Proposition~\ref{delta cover}.

\begin{proposition}\label{delta finest}
The congruence generated by $s_0s_1s_0$ and $s_1s_0s_1$ removes a shard $\Sigma$ if and only if one of the following conditions holds:  
\begin{enumerate}
\item[(i)] $\Sigma$ is contained in a hyperplane $H_t$ such that $t$ is $(i\,\,-j)(j\,\,-i)$ for some $i$ and $j$ with $2\le i<j\le n$.
\item[(ii)] $\Sigma$ is below the hyperplane $H_{(1\,\,-1)}$ and $\Sigma$ is contained in a hyperplane $H_t$ such that $t$ is $(1\,\,-j)(j\,\,-1)$ for some $j$ with $1<j\le n$.  
\item[(iii)] $\Sigma$ is above the hyperplane $H_{(1\,\,-1)}$ and $\Sigma$ is contained in a hyperplane $H_{(i\,\,-i)}$ for some $i>1$.
\end{enumerate}
\end{proposition}

Before proving Proposition~\ref{delta finest}, we verify that conditions (ii) and (iii) make sense.
Specifically, in both conditions, we rule out the possibility that $\Sigma$ is neither above nor below $H_{(1\,\,-1)}$.
Note that if $1<j\le n$ and $t=(1\,\,-j)(j\,\,-1)$, then the rank-two subarrangement containing $H_{(1\,\,-1)}$ and $H_t$ has basic hyperplanes $H_{(1\,\,-1)}$ and $H_{t'}$, where $t'=(1\,\,j)(-j\,\,-1)$.
Thus every hyperplane $H_t$, for $t$ as in condition (ii), is cut at $H_{(1\,\,-1)}$, and thus every shard in $H_t$ is either above or below $H_{(1\,\,-1)}$.
Similarly, every shard in $H_t$, for $t$ as in condition (iii), is either above or below $H_{(1\,\,-1)}$.

\begin{proof}
Suppose $\Sigma$ is a shard in a hyperplane that is cut by $H_{(1\,\,-1)}$, so that $\Sigma$ is either above or below $H_{(1\,\,-1)}$.
Then $\Sigma$ is above $H_{(1\,\,-1)}$ if and only if its associated join-irreducible element $\gamma$ has $-1$ in its one-line notation.
This occurs if and only if $-1$ is contained in the signed subset representing $\gamma$.
The shards specified by conditions (i)--(iii) in Proposition~\ref{delta finest} correspond, via join-irreducible elements, to signed subsets $A$ described respectively by the following conditions: 
\begin{enumerate}
\item[(i)] $A$ has $m<-1$, $M>1$ and $m\neq-M$.
\item[(ii)] $M=1$ and $m<-1$.
\item[(iii)] $-1\in A$ and $-M=m<-1$.
\end{enumerate}
For condition (ii), the condition is $(m,M)\in\set{(-1,j),(-j,1)}$ for some $1<j\le n$, and $-1\not\in A$.
However, the requirement that $-1\not\in A$ implies that $m\neq -1$, so $(m,M)=(-j,1)$, as indicated above.


Let $C$ be the set of signed subsets satisfying (i), (ii) or (iii).
We can describe $C$ more succinctly as the set of signed subsets satisfying both of the following conditions:
\begin{enumerate}
\item[(a)] $m<-1$, and 
\item[(b)] If $m=-M$ then $-1\in A$.
\end{enumerate}

We now show that no set $A_1\in C$ arrows a set $A_2\not\in C$.
Let $A_1\in C$ and $A_2\not\in C$.
Suppose $m_2\ge-1$.
Then (q1) and (q6) fail, because each would include the impossible assertion that $M_1<-m_2$.
Also, (q2)--(q4) fail because $m_1<-1$.
Since $M_1>0$, if (q5) holds, then $m_2=-1$ and $M_1=1$.
If in addition (r2) holds, then the fact that $m_2=-1\in A_2\cap(m_1,M_1)$ implies that $-1\in -A_2^c\cap(m_1,M_1)=A_1\cap(m_1,M_1)$.
But having $-1\in A_1$ contradicts the fact that $M_1=1$, and this contradiction rules out the possibility that (q5) and (r2) both hold.

We have ruled out all six possibilities in Theorem~\ref{B shard} in the case where $m_2\ge-1$.
If $m_2<-1$, then since $A_2\not\in C$, we must have $m_2=-M_2$ and $-1\not\in A_2$.
In particular, (q3)--(q6) fail.
As above, (q2) fails because $m_1<-1$.
If (q1) holds, then $m_1=-M_1$, so since $A_1\in C$, we have $-1\in A_1$.
In particular, $M_1>1$, so (r1) fails because $-1\in A_1$ but $-1\not\in A_2$.

Next, we show that any set in $C$, except  $\set{-2,-1,3,4,\ldots,n}$ and $\set{-2,3,4,\ldots,n}$, is arrowed to by another set in $C$.
Let $A_2\in C$.

\noindent
\textbf{Case 1:} $m_2=-M_2$.
Then $-1\in A_2$ because $A_2\in C$.
We can rule out the possibility that $m_2=-2$, because in this case, $A_2=\set{-2,-1,3,4,\ldots,n}$.
Let $A_1$ be $\set{-2,-1,3,4,\ldots,n}$, which is in $C$.
Then (q1) and (r1) hold, so $A_1\to A_2$.

\noindent
\textbf{Case 2:} $m_2<-M_2$.
Since $A_2\in C$, $m_2<-1$.
If $m_2=-2$, then $M_2=-1$ and $A_2=\set{-2,3,4,\ldots,n}$, which is ruled out by hypothesis.
Thus $m_2<-2$.

\noindent
\textbf{Subcase 2a:} $[m_2+1,-2]\cap A_2\neq\emptyset$.
Let $A_1=(A_2\cup\set{-m_2})\setminus\set{m_2}$.
Then $A_1\in C$ and (q3), (f1\,:\,$m_1$), and (r1) hold, so $A_1\to A_2$.

\noindent
\textbf{Subcase 2b:} $[m_2+1,-2]\cap A_2=\emptyset$ and $m_2<-M_2-1$.
Let $A_1=(A_2\cup\set{m_2+1,-m_2})\setminus\set{m_2,-m_2-1}$.
Then $A_1\in C$ and (q3), (f2\,:\,$m_1$), and (r1) hold.

\noindent
\textbf{Subcase 2c:} $[m_2+1,-2]\cap A_2=\emptyset$ and $m_2=-M_2-1$.
Since $m_2<-2$, $M_2>1$.
If $|A_2|<n-1$, then let $A_1=A_2\cup\set{M_2}$.
Then $A_1\in C$ and (q4), (f2\,:\,$M_1$), and (r1) hold.
If $|A_2|=n-1$ and $-1\in A_2$, then let $A_1=(A_2\cup\set{-M_2,-m_2})\setminus\set{m_2}$.
Then $m_1=-M_2<-1$, and $M_1=M_2$.
Since $-1\in A_1$, $A_1\in C$.
Also, (q3), (f3\,:\,$m_1$), and (r1) hold, so $A_1\to A_2$.
If $|A_2|=n-1$ and $-1\not\in A_2$, then $A_2=(\set{1,2,\ldots,n}\cup\set{m_2})\setminus\set{-m_2-1,-m_2}$, recalling that $M_2=-m_2-1$.
Let $A_1=(A_2\cup\set{M_2})\setminus\set{M_2-1}$.
Then $A_1\in C$ and (q4), (f1\,:\,$M_1$), and (r1) hold.

\noindent
\textbf{Case 3:} $m_2>-M_2$.

\noindent
\textbf{Subcase 3a:} $m_2>-M_2+1$.
Let $A_1=(A_2\cup\set{M_2})\setminus\set{M_2-1}$.
The element $M_2-1$ may or may not be an element of $A_2$, but since $-M_2+1<m_2$, we know that $-M_2+1\not\in A_2$.
Thus $M_1=M_2-1$, and $m_1=m_2$.
In particular, $A_1\in C$ and, furthermore, (q4) and (r1) hold.
Also, either (f1\,:\,$M_1$) or (f2\,:\,$M_1$) holds.

\noindent
\textbf{Subcase 3b:} $m_2=-M_2+1$.
If either $|A_2|<n-1$ or $-1\in A_2$, then let $A_2=A_1\cup\set{M_2}$.
Then $A_2\in C$ and (q4) and (r1) hold.
If $|A_2|<n-1$, then (f2\,:\,$M_1$) holds, and otherwise (f3\,:\,$M_1$) holds.
Finally, if $|A_2|=n-1$ and $-1\not\in A_2$, then let $A_1=(-A_2^c\cap(m_2,-m_2))\cup\set{m_2}\cup\set{M_2,M_2+1,\ldots,n}$.
Then $|A_1|=n$, $m_1=m_2$, and $M_1=-m_2$.
Conditions (q5), (f3\,:\,$-m_1$) and (r2) hold and $A_1\in C$.
\end{proof}

Combining Propositions~\ref{delta cover} and~\ref{delta finest}, the fibers of $\eta_\delta$ are a lattice congruence $\Theta_\delta$ on $B_n$ satisfying the second assertion of Theorem~\ref{delta thm}.
We now show that $\eta_\delta$ is order-preserving, that $\eta_\delta$ restricts to a bijection from bottom elements of $\Theta_\delta$-classes to permutations in $S_{n+1}$ and that the inverse of the restriction is order-preserving.

To see that $\eta_\delta$ is order-preserving, suppose $\pi\covered\tau$ in $B_n$.
If $1$ is an entry in the one-line notation of both $\pi$ and $\tau$, then $\eta_\delta$ coincides with $\eta_\sigma$ on $\pi$ and $\tau$, so $\eta_\delta$ preserves the order relation $\pi\le\tau$.
Similarly, if $-1$ is an entry in the one-line notation of both $\pi$ and $\tau$, then $\eta_\delta$ preserves the order relation $\pi\le\tau$ because $\eta_\nu$ is order-preserving.
Finally, if $1$ is an entry in the one-line notation of $\pi$ and $-1$ is an entry in the one-line notation of $\tau$, then $\eta_\delta(\pi)$ and $\eta_\delta(\tau)$ both have $1$ and $2$ adjacent but in different orders.
Otherwise, the two permutations agree, and we see that $\eta_\delta(\pi)\covered\eta_\delta(\tau)$. 
We have verified that $\eta_\delta$ is order-preserving.

Proposition~\ref{delta cover} leads to a characterization of the bottom elements of $\Theta_\delta$-classes.
A signed permutation $\pi$ whose one-line notation contains $1$ is a bottom element if and only if its one-line notation consists of a (possibly empty) sequence of negative entries followed by a sequence of positive entries (including $1$).
A signed permutation with $-1$ in its one-line notation is a bottom element if and only if it consists of a (possibly empty) sequence of positive entries, followed by a sequence of negative elements beginning with $-1$, and finally a (possibly empty) sequence of positive elements.
Notice that the bottom elements of $\Theta_\delta$ that have $1$ in their one-line notation map to permutations with $1$ preceding $2$, and the bottom elements of $\Theta_\delta$ that have $-1$ in their one-line notation map to permutations with $2$ preceding~$1$.

The inverse of the restriction of $\eta_\delta$ to bottom elements sends a permutation $\tau$ with $\tau_i=1$ and $\tau_j=2$, for $i<j$, to the signed permutation 
\[(-\tau_{i-1}+1)(-\tau_{i-2}+1)\cdots(-\tau_1+1)(\tau_{i+1}-1)(\tau_{i+2}-1)\cdots(\tau_{n+1}-1).\]
The inverse of the restriction sends a permutation $\tau_1\cdots\tau_{n+1}$ with $\tau_i=2$ and $\tau_j=1$, for $i<j$, to the signed permutation 
\[(\tau_{i+1}-1)\cdots(\tau_{j-1}-1)\,(-1)\,(-\tau_{i-1}+1)\cdots(-\tau_1+1)(\tau_{j+1}-1)\cdots(\tau_{n+1}-1).\]
It is now easily verified that the inverse is order-preserving.
The third assertion of Theorem~\ref{delta thm} follows as in the case of Theorem~\ref{simion thm}, and we have completed the proof of Theorem~\ref{delta thm}.

Now let $\eta_\ep:B_n\to S_{n+1}$ send $\pi\in B_n$ to the permutation obtained as follows:
If the one-line notation for $\pi$ contains the entry $1$, then extract the subsequence of $(-\pi_n)(-\pi_{n-1})\cdots(-\pi_1)\pi_1\cdots\pi_{n-1}\pi_n$ consisting of values greater than or equal to $-1$, change $-1$ to $0$, then add $1$ to each entry.
If the one-line notation for $\pi$ contains the entry $-1$, then construct a sequence 
\[(-\pi_n)(-\pi_{n-1})\cdots(-\pi_1)\,0\,\pi_1\cdots\pi_{n-1}\pi_n,\] 
extract the subsequence consisting of nonnegative entries, and add $1$ to each entry.
Thus $\eta_\ep$ is a hybrid of $\eta_\sigma$ and $\eta_\nu$ in exactly the opposite way that $\eta_\delta$ is a hybrid: $\eta_\ep(\pi)=\eta_\nu(\pi)$ if the one-line notation of $\pi$ contains $1$ and $\eta_\ep(\pi)=\eta_\sigma(\pi)$ if the one-line notation of $\pi$ contains $-1$.

\begin{theorem}\label{ep thm}
The map $\eta_\ep$ is a surjective lattice homomorphism from $B_n$ to $S_{n+1}$.
Its fibers constitute the congruence generated by $s_0s_1$ and $s_1s_0$.
Furthermore, $\eta_\ep$ is the unique surjective lattice homomorphism from $B_n$ to $S_{n+1}$ whose restriction to $(B_n)_{\set{s_0,s_1}}$ agrees with $\eta_\ep$.
\end{theorem}

Fortunately, to prove Theorem~\ref{ep thm}, we can appeal to Theorem~\ref{delta thm} and Proposition~\ref{dual cong}, to avoid any more tedious arguments about shard arrows for type~$B_n$.
Let $\neg:B_n\to B_n$ be the map that sends a signed permutation $\pi_1\cdots\pi_n$ to $(-\pi_1)\cdots(-\pi_n)$.
Let $\rev:S_{n+1}\to S_{n+1}$ be the map sending a permutation $\tau_1\cdots\tau_{n+1}$ to $\tau_{n+1}\cdots\tau_1$.
The following proposition is easily verified.

\begin{proposition}\label{delta ep dual}
The map $\eta_\ep$ sends a signed permutation $\pi$ to $\rev(\eta_\delta(\neg(\pi)))$.
\end{proposition}

The maps $\rev$ and $\neg$ both send a group element $w$ to $ww_0$, where $w_0$ is the longest element of the corresponding Coxeter group (the element $(n+1)n\cdots1$ in $S_{n+1}$ or the element $(-1)\cdots(-n)$ in $B_n$).
The map $w\mapsto ww_0$ is an anti-automorphism of the weak order on any finite Coxeter group.  
Thus the first assertion of Theorem~\ref{ep thm} follows from the first assertion of Theorem~\ref{delta thm}.
Furthermore, by Proposition~\ref{dual cong}, the second assertion of Theorem~\ref{ep thm} follows from the second assertion of Theorem~\ref{delta thm}, and the third assertion follows as usual.

\section{Homomorphisms from exceptional types}\label{exceptional sec}
In this section, we treat the remaining cases, where $(W,W')$ is $(F_4,A_4)$, $(H_3,A_3)$, $(H_3,B_3)$, $(H_4,A_4)$, or $(H_4,B_4)$. 
In each case, the stated theorem proves 
the first assertion of Theorem~\ref{existence}, while inspection of the proof establishes the second assertion of Theorem~\ref{existence} and completes the proof of Theorem~\ref{diagram uniqueness}.

\subsection*{Homomorphisms from type $F_4$}
Let $W$ be a Coxeter group of type $F_4$ and with $S=\set{p,q,r,s}$, $m(p,q)=3$, $m(q,r)=4$ and $m(r,s)=3$.
Figure~\ref{weakF4diagram} shows the order ideal in the weak order on $W$ that encodes these values of $m$.
This is comparable to Figure~\ref{weakB3diagram}.a, except that the square intervals that indicate where $m$ is $2$ (e.g. $m(p,r)=2$) are omitted.
\begin{figure}
\centerline{\includegraphics{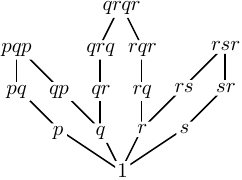}}
\caption{The Coxeter diagram of $F_4$ encoded as an order ideal}
\label{weakF4diagram}
\end{figure}
For $W'$ of type $A_4$ (i.e. $W'=S_5$), we identify $p,q,r,s$ with $s_1,s_2,s_3,s_4$ (in that order) to discuss homomorphisms fixing $S$ pointwise.

\begin{theorem}\label{F4 thm}
There are exactly four surjective lattice homomorphisms from $F_4$ to $A_4$ that fix $S$ pointwise:
For each choice of $\gamma_1\in\set{qr,qrq}$ and $\gamma_2\in\set{rq,rqr}$, there exists a unique such homomorphism whose associated congruence contracts $\gamma_1$ and $\gamma_2$.
\end{theorem}

\begin{proof}
Suppose $\eta:F_4\to A_4$ be a surjective homomorphism fixing $S$ pointwise.
By Proposition~\ref{diagram facts} and the discussion in Section~\ref{dihedral sec}, the congruence $\Theta$ on $F_4$ defined by the fibers of $\eta$ contracts exactly one of $qr$ and $qrq$ and exactly one of $rq$ and $rqr$.
In each of the four cases, we verify that there is a unique choice of $\eta$.

\noindent
\textbf{Case 1:}
\textit{$\Theta$ contracts $qr$ and $rq$.}
Computer calculations show that the quotient of $W$ modulo the congruence generated by $qr$ and $rq$ is isomorphic to the weak order on $A_4$.
Thus a unique $\eta$ exists in this case.

\noindent
\textbf{Case 2:}
\textit{$\Theta$ contracts $qrq$ and $rqr$.}
Computer calculations as in Case 1 (or combining Case 1 with Proposition~\ref{dual cong}) show that a unique $\eta$ exists.

\noindent
\textbf{Case 3:}
\textit{$\Theta$ contracts $qrq$ and $rq$.}
Proposition~\ref{diagram facts} implies that the restriction to $W_{\set{q,r,s}}$ is a surjective lattice homomorphism. 
Thus Theorem~\ref{nonhom thm} implies that the restriction of $\eta$ to $W_{\set{q,r,s}}$ agrees with $\eta_\nu$.
In particular, in addition to the join-irreducible elements of $W_{\set{q,r,s}}$ contracted by the congruence generated by $qrq$ and $rq$, the congruence $\Theta$ contracts $rqrs$ and $srqrs$.
Computer calculations show that the quotient of $W$ modulo the congruence generated by $qrq$, $rq$, $rqrs$, and $srqrs$ is isomorphic to the weak order on $A_4$.
Thus a unique $\eta$ exists.

\noindent
\textbf{Case 4:}
\textit{$\Theta$ contracts $qr$ and $rqr$.}
By an argument/calculation analogous to Case~3 (or by Case 3 and the diagram automorphism of $F_4$), a unique $\eta$ exists.
The congruence $\Theta$ is generated by $qr$, $rqr$, $qrqp$, and $pqrqp$.
\end{proof}

We have dealt with the final crystallographic case, and thus proved the following existence and uniqueness theorem.

\begin{theorem}\label{existence uniqueness crys}
Let $(W,S)$ and $(W',S)$ be finite crystallographic Coxeter systems with $m'(r,s)\le m(r,s)$ for each pair $r,s\in S$.
For each $r,s\in S$, fix a surjective homomorphism $\eta_\set{r,s}$ from $W_{\set{r,s}}$ to $W'_{\set{r,s}}$ with $\eta_\set{r,s}(r)=r$ and $\eta_\set{r,s}(s)=s$.
Then there is exactly one compressive homomorphism $\eta:W\to W'$ such that the restriction of $\eta$ to $W_{\set{r,s}}$ equals $\eta_\set{r,s}$ for each pair $r,s\in S$.
\end{theorem}

\subsection*{Homomorphisms from type $H_3$}
Now, let $W$ be a Coxeter group of type $H_3$ and with $S=\set{q,r,s}$, $m(q,r)=5$ and $m(r,s)=3$.
Identify $q,r,s$ with the generators $s_1,s_2,s_3$ of $A_3=S_4$.

\begin{theorem}\label{H3 A3 thm}
There are exactly nine surjective lattice homomorphisms from $H_3$ to $A_3$ that fix $S$ pointwise:
For each choice of distinct $\gamma_1,\gamma_2\in\set{qr,qrq,qrqr}$ and distinct $\gamma_3,\gamma_4\in\set{rq,rqr,rqrq}$, there exists a unique such homomorphism whose associated congruence contracts $\gamma_1$, $\gamma_2$, $\gamma_3$ and $\gamma_4$.
\end{theorem}

\begin{proof}
Suppose $\eta:H_3\to A_3$ is a surjective homomorphism, fixing $S$ pointwise, whose fibers define a congruence $\Theta$.
As in the proof of Theorem~\ref{F4 thm}, we consider the nine cases separately.

In each of Cases 1--6 below, computer calculations show that the quotient of $H_3$ modulo the congruence generated by the given join-irreducible elements is isomorphic to the weak order on $A_3$.
Thus in each of these cases, a unique $\eta$ exists.

\noindent
\textbf{Case 1:}
\textit{$\Theta$ contracts $qr$, $qrq$, $rq$, and $rqr$.}

\noindent
\textbf{Case 2:}
\textit{$\Theta$ contracts $qr$, $qrq$, $rq$, and $rqrq$.}

\noindent
\textbf{Case 3:}
\textit{$\Theta$ contracts $qr$, $qrq$, $rqr$, and $rqrq$.}

\noindent
\textbf{Case 4:}
\textit{$\Theta$ contracts $qr$, $qrqr$, $rq$, and $rqrq$.}

\noindent
\textbf{Case 5:}
\textit{$\Theta$ contracts $qr$, $qrqr$, $rqr$, and $rqrq$.}

\noindent
\textbf{Case 6:}
\textit{$\Theta$ contracts $qrq$, $qrqr$, $rqr$, and $rqrq$.}

\noindent
\textbf{Case 7:}
\textit{$\Theta$ contracts $qrq$, $qrqr$, $rq$, and $rqr$.}
In this case, the quotient of $H_3$ modulo the congruence $\Theta'$ generated by the given join-irreducible elements is a poset with 28 elements.
Inspection of $\Irr(\Con(H_3/\Theta'))$ reveals a unique way to contract additional join-irreducible elements so as to obtain $\Irr(\Con(A_3))$.
(Cf. Figure~\ref{IrrConB3 squashed fig} and the surrounding discussion in Section~\ref{nonhom sec}.)
The additional join-irreducible elements to be contracted are $rqrqsrq$ and $srqrqsrq$.
Computer calculations confirm that the quotient of $H_3$ modulo the congruence generated by $qrq$, $qrqr$, $rq$, $rqr$, $rqrqsrq$, and $srqrqsrq$ is indeed isomorphic to $A_3$, so there exists a unique $\eta$ in this case.

\noindent
\textbf{Case 8:}
\textit{$\Theta$ contracts $qr$, $qrqr$, $rq$, and $rqr$.}
Computer calculations show that the quotient of $H_3$ modulo the congruence generated by $qr$, $qrqr$, $rq$, $rqr$, 
$rqrqsrqr$, and $srqrqsrqr$  is isomorphic to $A_3$.
As in Case 7, inspection of the poset of irreducibles of the congruence lattice of the quotient reveals that no other possibilities exist.
Thus $\eta$ exists and is unique in this case.

\noindent
\textbf{Case 9:}
\textit{$\Theta$ contracts $qrq$, $qrqr$, $rq$, and $rqrq$.}
The existence and uniqueness of $\eta$ follows from Case 7 and Proposition~\ref{dual cong}.
Alternately, the calculation can be made directly as in Cases 7--8, realizing $\Theta$ as the congruence generated by $qrq$, $qrqr$, $rq$, $rqrq$, $rqrs$, and $srqrs$.
\end{proof}

\begin{theorem}\label{H3 B3 thm}
There are exactly eight surjective lattice homomorphisms from $H_3$ to $B_3$:
There is no surjective lattice homomorphism $\eta:H_3\to B_3$ whose associated congruence contracts $qr$ and $rqrq$.
For every other choice of $\gamma_1\in\set{qr,qrq,qrqr}$ and $\gamma_2\in\set{rq,rqr,rqrq}$, there exists a unique such homomorphism whose associated congruence contracts $\gamma_1$ and $\gamma_2$.
\end{theorem}

\begin{proof}
Suppose $\eta:H_3\to B_3$ is a surjective homomorphism whose fibers define a congruence $\Theta$.
We again proceed by cases.

\noindent
\textbf{Case 1:}
\textit{$\Theta$ contracts $qrq$ and $rqr$.}
Computer calculations show that the quotient of $H_3$ modulo the congruence generated by $qrq$ and $rqr$ is isomorphic to the weak order on $B_3$.
Thus a unique $\eta$ exists.

\noindent
Case 1 is the only case in which $\Theta$ is homogeneous of degree $2$.
Cases 2--8 proceed just like Cases 7--9 in the proof of Theorem~\ref{H3 A3 thm}, except that ``inspection of the poset of irreducibles of the congruence lattice of the quotient'' is automated.
In each case, there exists a unique set of additional join-irreducibles required to generate $\Theta$, and these generators are given in parentheses.
Some of these cases can also be obtained from each other by Proposition~\ref{dual cong}.

\noindent
\textbf{Case 2:}
\textit{$\Theta$ contracts $qr$ and $rq$}
($qr$, $rq$, and $qrqsrqrs$).

\noindent
\textbf{Case 3:}
\textit{$\Theta$ contracts $qr$ and $rqr$}
($qr$, $rqr$, $qrqsrqrs$, $rqrqsrqrs$, and $srqrqsrqrs$).

\noindent
\textbf{Case 4:}
\textit{$\Theta$ contracts $qrq$ and $rq$}
($qrq$, $rq$, $qsrqrs$, $rqrqsr$, $srqrqsr$, $rqrs$, $qrqrs$, and $srqrs$).

\noindent
\textbf{Case 5:}
\textit{$\Theta$ contracts $qrq$ and $rqrq$}
($qrq$, $rqrq$, $qsrqrs$, $qrqrs$ and $rqsrqrs$).

\noindent
\textbf{Case 6:}
\textit{$\Theta$ contracts $qrqr$ and $rq$}
($qrqr$, $rq$, $qsrqrqsrqr$, $rqrqsrqr$, $qrqrqsrqr$, $srqrqsrqr$, $rqrqsr$, $srqrqsr$, $rqrs$, and $srqrs$).

\noindent
\textbf{Case 7:}
\textit{$\Theta$ contracts $qrqr$ and $rqr$}
($qrqr$, $rqr$, $srqrqsrqrs$, $qsrqrqsrqr$, $rqrqsrqr$, $qrqrqsrqr$, $srqrqsrqr$, and $rqrqsrqrs$).

\noindent
\textbf{Case 8:}
\textit{$\Theta$ contracts $qrqr$ and $rqrq$}
($qrqr$, $rqrq$, and $rqsrqrs$).

\noindent
\textbf{Case 9:}
\textit{$\Theta$ contracts $qr$ and $rqrq$.}
A computation shows that the quotient of $H_3$ modulo the congruence generated by $qr$ and $rqrq$ is not isomorphic to $B_3$.
Furthermore, this quotient has 48 elements, the same number as $B_3$.
Thus it is impossible to obtain $B_3$ by contracting additional join-irreducible elements.
\end{proof}

\subsection*{Homomorphisms from type $H_4$}
Finally, let $W$ be a Coxeter group of type $H_4$ and with $S=\set{q,r,s,t}$, $m(q,r)=5$, $m(r,s)=3$ and $m(s,t)=3$.
The classification of surjective homomorphisms from $H_4$ to $A_4$ and from $H_4$ to $B_4$ exactly follows the classification of surjective homomorphisms from $H_3$ to $A_3$ and from $H_3$ to $B_3$, as we now explain.

Let $\eta$ be any surjective homomorphism $\eta$ from $H_4$ to $A_4$ or $B_4$ with associated congruence $\Theta$.
By Proposition~\ref{diagram facts}, the restriction $\eta'$ of $\eta$ to the standard parabolic subgroup $W_{\set{q,r,s}}$ (of type $H_3$) is a homomorphism from $H_3$ to $A_3$ or $B_3$.
Thus $\eta'$ is described by Theorem~\ref{H3 A3 thm} or~\ref{H3 B3 thm}.
The congruence $\Theta'$ associated to $\eta'$ is the restriction of $\Theta$ to $W_{\set{q,r,s}}$.
The proofs of Theorems~\ref{H3 A3 thm} and~\ref{H3 B3 thm}, determine, for each surjective homomorphism, a set $\Gamma$ of join-irreducibles that generate the associated congruence.
Since $\Theta'$ agrees with $\Theta$ on $W_{\set{q,r,s}}$, the congruence associated to $\eta$ must also contract the join-irreducibles in $\Gamma$.
In each case, computer calculations show that the quotient of $H_4$ modulo the congruence generated by $\Gamma$ is isomorphic to $A_4$ or $B_4$.
This shows that for each surjective homomorphisms from $H_3$, there is a unique surjective homomorphisms from $H_4$.
Furthermore, Theorem~\ref{H3 B3 thm} and Proposition~\ref{diagram facts} imply that there is no surjective lattice homomorphism from $H_4$ to $B_4$ whose associated congruence contracts $qr$ and $rqrq$.
Thus we have the following theorems:

\begin{theorem}\label{H4 A4 thm}
There are exactly nine surjective lattice homomorphisms from $H_4$ to $A_4$ that fix $S$ pointwise:
For each choice of distinct $\gamma_1,\gamma_2\in\set{qr,qrq,qrqr}$ and distinct $\gamma_3,\gamma_4\in\set{rq,rqr,rqrq}$, there exists a unique such homomorphism whose associated congruence contracts $\gamma_1$, $\gamma_2$, $\gamma_3$ and $\gamma_4$.
\end{theorem}

\begin{theorem}\label{H4 B4 thm}
There are exactly eight surjective lattice homomorphisms from $H_4$ to $B_4$:
There is no surjective lattice homomorphism $\eta:H_4\to B_4$ whose associated congruence contracts $qr$ and $rqrq$.
For every other choice of $\gamma_1\in\set{qr,qrq,qrqr}$ and $\gamma_2\in\set{rq,rqr,rqrq}$, there exists a unique such homomorphism whose associated congruence contracts $\gamma_1$ and $\gamma_2$.
\end{theorem}

\section{Lattice homomorphisms between Cambrian lattices}\label{camb sec}
In this section, we show how the results of this paper on lattice homomorphisms between weak orders imply similar results on lattice homomorphisms between Cambrian lattices.

A \newword{Coxeter element} of a Coxeter group $W$ is an element $c$ that can be written in the form $s_1s_2\cdots s_n$, where $s_1,s_2,\ldots,s_n$ are the elements of $S$.
There may be several total orders $s_1,s_2,\ldots,s_n$ on $S$ whose product is $c$.
These differ only by commutations of commuting elements of $S$.
In particular, for every edge $r$---$s$ in the Coxeter diagram for $W$ (i.e.\ for every pair $r,s$ with $m(r,s)>2$), either $r$ precedes $s$ in every reduced word for $c$ or $s$ precedes $r$ in every reduced word for $c$.
We use the shorthand ``$r$ precedes $s$ in $c$'' or ``$s$ precedes $r$ in $c$'' for these possibilities.

The initial data defining a Cambrian lattice are a finite Coxeter group $W$ and a Coxeter element $c$.
The \newword{Cambrian congruence} is the homogeneous degree-$2$ congruence $\Theta_c$ on $W$ generated by the join-irreducible elements $\alt_k(s,r)$ for every pair $r,s\in S$ such that $r$ precedes $s$ in $c$ and every $k$ from $2$ to $m(r,s)-1$.
Here, as in Section~\ref{erase edge sec}, the notation $\alt_k(s,r)$ stands for the word with $k$ letters, starting with~$s$ and then alternating $r$, $s$, $r$, etc.
Equivalently, $\Theta_c$ is the finest congruence with $s\equiv\alt_{m(r,s)-1}(s,r)$ for every pair $r,s\in S$ such that $r$ precedes $s$ in $c$.
The quotient $W/\Theta_c$ is the \newword{Cambrian lattice}.
The Cambrian lattice is isomorphic to the subposet of $W$ induced by the bottom elements of $\Theta_c$-classes.
To distinguish between the two objects, we write $W/\Theta_c$ for the Cambrian lattice constructed as a quotient of $W$ and write $\Camb(W,c)$ for the Cambrian lattice constructed as the subposet of $W$ induced by bottom elements of $\Theta_c$-classes.

The key result of this section is the following theorem, which will allow us to completely classify surjective homomorphisms between Cambrian lattices, using the classification results on surjective homomorphisms between weak orders.
(These classification results are Theorems~\ref{camb para factor}, \ref{camb exist unique}, and~\ref{camb diagram}.)

\begin{theorem}\label{key camb}
Let $\eta:W\to W'$ be a surjective lattice homomorphism whose associated congruence $\Psi$ is generated by a set $\Gamma$ of join-irreducible elements.
Let $c=s_1s_2\cdots s_n$ be a Coxeter element of $W$ and let $c'=\eta(s_1)\eta(s_2)\cdots\eta(s_n)\in W'$.
Then the restriction of $\eta$ is a surjective lattice homomorphism from $\Camb(W,c)$ to $\Camb(W',c')$.
The associated congruence is generated by $\Gamma\cap\Camb(W,c)$.
\end{theorem}
To understand Theorem~\ref{key camb} correctly, one should keep in mind that $\eta$ is a lattice homomorphism, not a group homomorphism, so that $c'$ need not be equal to $\eta(c)$.
(For example, in the case $n=2$ of Theorem~\ref{simion thm}, the lattice homomorphism $\eta_\sigma$ sends the Coxeter element $s_0s_1$ to $s_1$.)
The element $c'$ is a Coxeter element of $W'$ in light of Proposition~\ref{basic facts} and Theorem~\ref{para factor}.

Before assembling the tools necessary to prove Theorem~\ref{key camb}, we illustrate the theorem by extending Example~\ref{miraculous}.

\begin{example}\label{camb hom ex}
Recall that Figure~\ref{B3toA3} indicates a congruence $\Psi$ on $B_3$ such that the quotient $B_3/\Psi$ is isomorphic to $S_4$.
As mentioned just after Theorem~\ref{delta thm}, the corresponding surjective homomorphism is $\eta_\delta:B_3\to S_4$.
The congruence $\Psi$ is generated by $\Gamma=\set{s_0s_1s_0,s_1s_0s_1}$.

Let $c$ be the Coxeter element $s_0s_1s_2$ of $B_3$.
The congruence $\Theta_c$ on $B_3$ is generated by $\set{s_1s_0,s_1s_0s_1,s_2s_1}$, as illustrated in Figure~\ref{cambriandiagram}.a.
\begin{figure}
\begin{tabular}{ccc}
\scalebox{.85}{\includegraphics{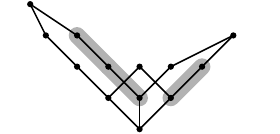}}&&
\scalebox{.85}{\includegraphics{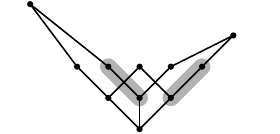}}\\
(a)&&(b)
\end{tabular}
\caption{a: Edge contractions that generate the $s_0s_1s_2$-Cambrian congruence on $B_3$.
b: Edge contractions that generate the $s_1s_2s_3$-Cambrian congruence on $S_4$.}
\label{cambriandiagram}
\end{figure}
Let $c'$ be the Coxeter element $s_1s_2s_3$ of $S_4$.
The Cambrian congruence $\Theta_{c'}$ on $S_4$ is generated by the set $\set{s_2s_1,s_3s_2}$, as illustrated in Figure~\ref{cambriandiagram}.b.
The set $\set{s_2s_1,s_3s_2}$ is obtained by applying $\eta_\delta$ to each element of $\set{s_1s_0,s_1s_0s_1,s_2s_1}$ that is not contracted by $\Psi$.

Figure~\ref{camb B3toA3 cong} shows the Cambrian congruence $\Theta_c$ on $B_3$ and the Cambrian congruence $\Theta_{c'}$ on $S_n$.
\begin{figure}
\begin{tabular}{ccc}
\scalebox{.85}{\includegraphics{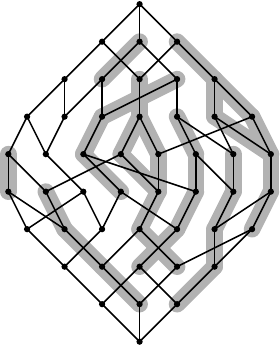}}&&\scalebox{.85}{\includegraphics{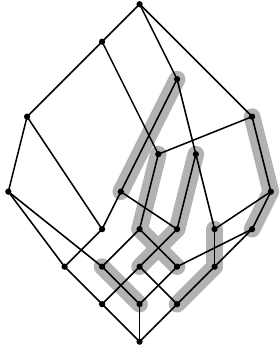}}\\
(a)&&(b)
\end{tabular}
\caption{a:  The Cambrian congruence $\Theta_c$ on $B_3$.  b:  The Cambrian congruence $\Theta_{c'}$ on $S_n$.}
\label{camb B3toA3 cong}
\end{figure}
\begin{figure}
\begin{tabular}{ccc}
\scalebox{.85}{\includegraphics{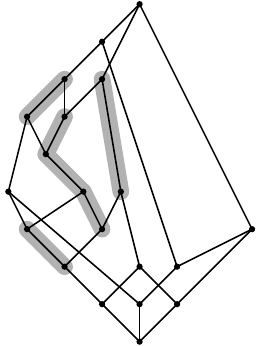}}&&\scalebox{.85}{\includegraphics{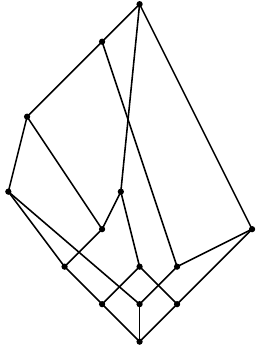}}\\
(a)&&(b)
\end{tabular}
\caption{a:  The Cambrian lattice $\Camb(B_3,c)\cong B_3/\Theta_c$. \qquad b:  The Cambrian lattice $\Camb(S_4,c')\cong S_4/\Theta_{c'}$.}
\label{camb B3toA3}
\end{figure}
Here, $S_4$ is represented, as in Figure~\ref{B3toA3}.b, as the subposet of $B_3$ induced by bottom elements of the congruence $\Psi$.
Figure~\ref{camb B3toA3} shows the Cambrian lattice $\Camb(B_3,c)$ and the Cambrian lattice $\Camb(S_4,c')$.
The shaded edges of $\Camb(B_3,c)$ indicate the congruence on $\Camb(B_3,c)$ whose quotient is $\Camb(S_4,c')$.
This congruence is generated by the join-irreducible element $s_0s_1s_0$, because $\Gamma\cap\Camb(B_3,c)=\set{s_0s_1s_0}$.
\end{example}

The proof of Theorem~\ref{key camb} depends on general lattice-theoretic results, but also on special properties of the subposet $\Camb(W,c)$ of $W$.
To explain these properties, we recall from ~\cite{sort_camb} the characterization of bottom elements of $\Theta_c$-classes as the $c$-sortable elements of $W$.

To define $c$-sortable elements, we fix a reduced word $s_1\cdots s_n$ for $c$ and construct, for any element $w\in W$, a canonical reduced word for $w$.
Write a half-infinite word $c^\infty=s_1\cdots s_n|s_1\cdots s_n|s_1\cdots s_n|\ldots$,
where the symbols ``$\,|\,$'' are dividers that mark the locations where the sequence $s_1\cdots s_n$ begins again.
Out of all subwords of $c^\infty$ that are reduced words for $w$, the \newword{$(s_1\cdots s_n)$-sorting word} for $w$ is the one that is lexicographically leftmost, as a sequence of positions in $c^\infty$.
For convenience, when we write $(s_1\cdots s_n)$-sorting words, we include the dividers that occur between letters in the lexicographically leftmost reduced subword.
Thus for example, if $W=S_4$, the $(s_1s_2s_3)$-sorting word for the longest element is $s_1s_2s_3|s_1s_2|s_1$, and the $(s_2s_1s_3)$-sorting word for the longest element is $s_2s_1s_3|s_2s_1s_3$.
Each $(s_1\cdots s_n)$-sorting word defines a sequence of subsets of $S$, by taking the elements between successive dividers, and this sequence depends only on $c$, not on $s_1\cdots s_n$.
An element $w$ is \newword{$c$-sortable} if this sequence of subsets is weakly decreasing.

The following is a combination of \cite[Theorems~1.1, 1.2, 1.4]{sort_camb}.
\begin{theorem}\label{sort_camb thm}
An element of $W$ is the bottom element of its $\Theta_c$-class if and only if it is $c$-sortable.
The $c$-sortable elements constitute a sublattice of the weak order.
\end{theorem}
In general, the set of bottom elements of a congruence is a join-sublattice but need not be a sublattice.

As a special case of a general lattice-theoretic result (found, for example, in \cite[Proposition~9-5.11]{regions9}), a $c$-sortable element $v$ represents a join-irreducible element of the Cambrian lattice $W/\Theta_c$ if and only if $v$ is join-irreducible as an element of $W$.  The following simple lemma will be used in the proof of Theorem~\ref{key camb}.

\begin{lemma}\label{sort j*}
If $j$ is a $c$-sortable join-irreducible element and $j_*$ is the unique element covered by $j$, then $j_*$ is $c$-sortable.
\end{lemma}
\begin{proof}
If $c=s_1\cdots s_n$ and $a_1a_2\cdots a_k$ is the $(s_1\cdots s_n)$-sorting word for $j$, then $a_1a_2\cdots a_{k-1}$ is the $(s_1\cdots s_n)$-sorting word for a $c$-sortable element $x$ covered by~$j$.
But $j_*$ is the unique element covered by $j$, so $j_*=x$, which is $c$-sortable.
\end{proof}	

We need one of the standard Isomorphism Theorems for lattices.  
(See, for example, \cite[Theorem~9-5.22]{regions9}.)
This is easily proved directly, or follows as a special case of the same Isomorphism Theorem in universal algebra.
The notation $[x]_\Theta$ stands for the $\Theta$-class of $x$.
\begin{theorem}\label{3 isom}
Let $L$ be a finite lattice and let $\Theta$ and $\Psi$ be congruences on $L$ such that $\Psi$ refines $\Theta$.
Define a relation $\Theta/\Psi$ on $L/\Psi$ by setting $[x]_\Psi\equiv[y]_\Psi$ modulo $\Theta/\Psi$ if and only if $x\equiv y$ modulo $\Theta$.
Then $\Theta/\Psi$ is a congruence and the map $\beta:L/\Theta\to(L/\Psi)/(\Theta/\Psi)$ sending $[x]_\Theta$ to the set of $\Psi$-classes contained in $[x]_\Theta$ is an isomorphism.
The inverse isomorphism sends a $\Theta/\Psi$-class $C$ in $(L/\Psi)/(\Theta/\Psi)$ to the union of the $\Psi$-classes in $C$.
\end{theorem}

The following lemma will also be useful.

\begin{lemma}\label{Psi and tilde gen}
Let $L$ be a finite lattice and let $\Gamma$ and $\tilde\Gamma$ be sets of join-irreducible elements in $L$.
Let $\Psi$ be the congruence on $L$ generated by $\Gamma$ and let $I$ be the set of join-irreducible elements contracted by $\Psi$.
Define $\tilde\Psi$ and $\tilde{I}$ similarly.
Then the congruence $(\Psi\join\tilde\Psi)/\Psi$ on $L/\Psi$ is generated by the set $\set{[j]_{\Psi}:j\in \tilde\Gamma\setminus I}$ of join-irreducible elements of $L/\Psi$.
\end{lemma} 
\begin{proof}
Identifying join-irreducible elements of $L$ with join-irreducible congruences as before, the sets $I$ and $\tilde{I}$ are the ideals that $\Gamma$ and $\tilde\Gamma$ generate in $\Irr(\Con(L))$.
The join-irreducible elements in the quotient $L/\Psi$ are exactly the elements $[j]_{\Psi}$ where $j\in(\Irr(L)\setminus I)$.
Since the set of join-irreducible elements contracted by $\Psi\join\tilde\Psi$ is $I\cup \tilde{I}$, each $[j]_{\Psi}$ is contracted by the congruence $(\Psi\join\tilde\Psi)/\Psi$ if and only if $j$ is in~$\tilde{I}$.
Furthermore, a join-irreducible element $j\in(\Irr(L)\setminus I)$ is in $\tilde{I}$ if and only if it is below some element of $\tilde\Gamma\setminus I$.
We see that $(\Psi\join\tilde\Psi)/\Psi$ is the congruence on $L/\Psi$ generated by the set $\set{[j]_{\Psi}:j\in \tilde\Gamma\setminus I}$.
\end{proof}

We now prove our key theorem.

\begin{proof}[Proof of Theorem~\ref{key camb}]
The quotient $W/\Psi$ is isomorphic to $W'$.
Let $I$ be the set of join-irreducible elements contracted by $\Psi$.
Let $\tilde\Gamma$ be the generating set that was used to define the Cambrian congruence $\Theta_c$.
By Lemma~\ref{Psi and tilde gen}, the set $\set{[j]_{\Psi}:j\in \tilde\Gamma\setminus I}$ generates the congruence $(\Psi\join\Theta_c)/\Psi$ on $W/\Psi\cong W'$.
This congruence corresponds to the Cambrian congruence $\Theta_{c'}$ on $W'$, so $\Camb(W',c')$ is isomorphic to $(W/\Psi)/[(\Psi\join\Theta_c)/\Psi]$, which, by Theorem~\ref{3 isom}, is isomorphic to $W/(\Phi\join\Theta_c)$.

Also by Theorem~\ref{3 isom}, $W/(\Phi\join\Theta_c)$ is isomorphic to $(W/\Theta_c)/[(\Phi\join\Theta_c)/\Theta_c]$, and Lemma~\ref{Psi and tilde gen} says that $(\Phi\join\Theta_c)/\Theta_c$ is generated by $\set{[j]_{\Theta_c}:j\in(\Gamma\cap\Camb(W,c))}$.
In particular, there is a surjective homomorphism from $W/\Theta_c$ to $W/(\Psi\join\Theta_c)$ whose associated congruence is generated by $\set{[j]_{\Theta_c}:j\in(\Gamma\cap\Camb(W,c))}$.
Since $x\mapsto[x]_{\Theta_c}$ is an isomorphism from $\Camb(W,c)$ to $W/\Theta_c$, and since $W/(\Psi\join\Theta_c)$ is isomorphic to $\Camb(W',c')$, we conclude that there is a surjective homomorphism from $\Camb(W,c)$ to $\Camb(W',c')$ whose associated congruence is generated by $\Gamma\cap\Camb(W,c)$.
We will show that this homomorphism is the restriction of $\eta$.

If $x\in\Camb(W',c')$, then $x$ is the bottom element of $[x]_{\Theta_{c'}}$.
Thus $\eta^{-1}(x)$ is a $\Psi$-class in $W$ containing the bottom element $y$ of a $(\Psi\join\Theta_c)$-class.
In particular, $y$ is the bottom element of $[y]_{\Theta_c}$, or in other words $y\in\Camb(W,c)$.
Since $\eta(y)=x$, we have shown that $\Camb(W',c')\subseteq\eta(\Camb(W,c))$.

Since $\Camb(W,c)$ is a sublattice of $W$ and $\eta$ is a homomorphism, $\eta(\Camb(W,c))$ is a sublattice of $W'$, and the restriction of $\eta$ to $\Camb(W,c)$ is a lattice homomorphism from $\Camb(W,c)$ to $\eta(\Camb(W,c))$.
Let $j\in\Gamma\cap\Camb(W,c)$.
Then because $j\in\Gamma$, $\eta$ contracts $j$, or in other words $\eta(j)=\eta(j_*)$, where $j_*$ is the unique element of $W$ covered by $j$.
By Lemma~\ref{sort j*}, the element $j_*$ is also in $\Camb(W,c)$, so $j_*$ is also the unique element of $\Camb(W,c)$ covered by $j$.
Thus the restriction of $\eta$ to $\Camb(W,c)$ also contracts $j$ in $\Camb(W,c)$.
Since there exists a surjective lattice homomorphism from $\Camb(W,c)$ to $\Camb(W',c')$ whose associated congruence is generated by $\Gamma\cap\Camb(W,c)$, the image of the restriction of $\eta$ can be no larger than $\Camb(W',c')$, and we conclude that $\eta(\Camb(W,c))=\Camb(W,c)$.
We have shown that the restriction of $\eta$ is a surjective homomorphism from $\Camb(W,c)$ to $\Camb(W',c')$ whose associated congruence is generated by $\Gamma\cap\Camb(W,c)$.
\end{proof}

With Theorem~\ref{key camb} in hand, we prove several facts that together constitute a detailed classification of surjective homomorphisms between Cambrian lattices.

An \newword{oriented Coxeter diagram} is a directed graph (with labels on some edges) defined by choosing an orientation of each edge in a Coxeter diagram.
Orientations of the Coxeter diagram of a finite Coxeter group $W$ are in bijection with Coxeter elements of $W$.
There is disagreement in the literature about the convention for this bijection.
The definition of Cambrian lattices given here agrees with the definition in~\cite{cambrian} if we take the following convention, which also agrees with the convention of \cite{sortable,sort_camb}:
To obtain a Coxeter element from an oriented diagram, we require, for each directed edge $s\to t$, that $s$ precedes $t$ in every expression for $c$.
However, the opposite convention is also common.

An \newword{oriented diagram homomorphism} starts with the oriented Coxeter diagram encoding a Coxeter system $(W,S)$ and a choice of Coxeter element $c$, then deletes vertices, decreases labels on directed edges, and/or erases directed edges, and relabels the vertices to obtain the oriented Coxeter diagram of some Coxeter system $(W',S')$ and choice $c'$ of Coxeter element.
We prove the following result and several more detailed results.

\begin{theorem}\label{main camb}
Given a finite Coxeter system $(W,S)$ with a choice $c$ of Coxeter element and another Coxeter system $(W',S')$ with a choice $c'$ of Coxeter element, there exists a surjective lattice homomorphism from $\Camb(W,c)$ to $\Camb(W',c')$ if and only if there exists an oriented diagram homomorphism from the oriented diagram for $(W,S)$ and $c$ to the oriented diagram for $(W',S')$ and $c'$.
\end{theorem}

The proof of Theorem~\ref{main camb} begins with a factorization result analogous to Theorem~\ref{para factor}.
Given $J\subseteq S$ and a Coxeter element $c$ of $W$, the \newword{restriction of $c$ to $W_J$} is the Coxeter element $\tilde c$ of $W_J$ obtained by deleting the letters in $S\setminus J$ from a reduced word for $c$.
(Typically, the restriction is not equal to $c_J$.)
Recall from Section~\ref{delete vert sec} the definition of the parabolic homomorphism $\eta_J$.
Recall from Section~\ref{intro} that a compressive homomorphism is a surjective homomorphism that restricts to a bijection between sets of atoms.
For Cambrian lattices, this is a surjective lattice homomorphism $\Camb(W,c)\to\Camb(W',c')$ that restricts to a bijection between $S$ and $S'$.

\begin{theorem}\label{camb para factor}
Suppose $(W,S)$ and $(W',S')$ are finite Coxeter systems and $c$ and $c'$ are Coxeter elements of $W$ and $W'$. 
Suppose $\eta:\Camb(W,c)\to\Camb(W',c')$ is a surjective lattice homomorphism and let $J\subseteq S$ be $\set{s\in S:\eta(s)\neq 1'}$.
Then $\eta$ factors as $\eta|_{\Camb(W_J,\tilde{c})}\circ(\eta_J)|_{\Camb(W,c)}$, where $\tilde{c}$ is the restriction of $c$ to $W_J$.
The map $\eta|_{\Camb(W_J,\tilde{c})}$ is a compressive homomorphism.
\end{theorem}

As was the case for the weak order, the task is to understand compressive homomorphisms between Cambrian lattices.
The following proposition is proved below as a special case of part of Proposition~\ref{camb basic facts}.
Together with Theorem~\ref{camb para factor}, it implies that every surjective lattice homomorphism between Cambrian lattices determines an oriented diagram homomorphism.
\begin{proposition}\label{oriented}
Suppose $\eta:\Camb(W,c)\to\Camb(W',c')$ is a compressive homomorphism.
Then $m'(\eta(r),\eta(s))\le m(r,s)$ for each pair $r,s\in S$.
Also, if $s_1\cdots s_n$ is a reduced word for $c$, then $c'=\eta(s_1)\cdots\eta(s_n)$.
\end{proposition}

As we did for the weak order, in studying compressive homomorphisms between Cambrian lattices, we may as well take $S'=S$ and let $\eta$ fix each element of $S$.
We will prove the following characterization of compressive homomorphisms.

\begin{theorem}\label{camb exist unique}
Suppose $(W,S)$ and $(W',S)$ are finite Coxeter systems and suppose that $m'(r,s)\le m(r,s)$ for each pair $r,s\in S$.
Given a Coxeter element $c$ of $W$ with reduced word $s_1\cdots s_n$, let $c'$ be the element of $W'$ with reduced word $s_1\cdots s_n$.
Let $\tilde\Gamma$ be a set of join-irreducible elements obtained by choosing, for each $r,s\in S$ with $m'(r,s)<m(r,s)$ such that $r$ precedes $s$ in $c$, exactly $m(r,s)-m(r,s)$ join-irreducible elements in $\set{\alt_k(r,s):k=2,3,\ldots,m(r,s)-1}$.
\begin{enumerate}
\item \label{exist unique tildeGamma}
There exists a unique homomorphism $\eta:\Camb(W,c)\to\Camb(W',c')$ that is compressive and fixes $S$ pointwise and whose associated congruence $\Theta$ contracts all of elements of~$\tilde\Gamma$.
\item \label{generated}
$\Theta$ is generated by $\tilde\Gamma$.
\item \label{every arises}
Every compressive homomorphism from $\Camb(W,c)$ to $\Camb(W',c')$ that fixes $S$ pointwise arises in this manner, for some choice of $\tilde\Gamma$.
\end{enumerate}
\end{theorem}

Theorem~\ref{key camb} implies that, given any compressive homomorphism $\tilde\eta:W\to W'$, there is a choice of $\tilde\Gamma$ in Theorem~\ref{camb exist unique} such that every element of $\tilde\Gamma$ is contracted by the congruence associated to $\tilde\eta$.
The homomorphism $\eta$ thus arising from Theorem~\ref{camb exist unique} is the restriction of $\tilde\eta$.
The following theorem is a form of converse to these statements.

\begin{theorem}\label{camb diagram}
Suppose $(W,S)$ and $(W',S)$ are finite Coxeter systems and suppose that $\eta:\Camb(W,c)\to\Camb(W',c')$ is an compressive homomorphism.
\begin{enumerate}
\item \label{exists diagram}
There exists a compressive homomorphism from $W$ to $W'$ whose restriction to $\Camb(W,c)$ is $\eta$.
\item \label{any diagram}
Let $\tilde\Gamma$ be the set of join-irreducible elements that generates the congruence associated to $\eta$.
Given a compressive homomorphism $\tilde\eta:W\to W'$ such that all elements of $\tilde\Gamma$ are contracted by the congruence associated to $\tilde\eta$, the restriction of $\tilde\eta$ to $\Camb(W,c)$ is $\eta$.
\end{enumerate}
\end{theorem}

We begin our proof of these classification results by pointing out some basic facts on surjective homomorphisms between Cambrian lattices, in analogy with Proposition~\ref{basic facts}.
Proposition~\ref{oriented} is a special case of (5) and (6) in the following proposition.

\begin{proposition}\label{camb basic facts}
Let $\eta:\Camb(W,c)\to \Camb(W',c')$ be a surjective lattice homomorphism.
Then
\begin{enumerate}
\item $\eta(1)=1'$.
\item $S'\subseteq\eta(S)\subseteq (S'\cup\set{1'})$.
\item If $r$ and $s$ are distinct elements of $S$ with $\eta(r)=\eta(s)$, then $\eta(r)=\eta(s)=1'$.
\item \label{restrict camb}
If $J\subseteq S$, then $\eta$ restricts to a surjective homomorphism from $\Camb(W_J,\tilde{c})$  to $\Camb(W'_{\eta(J)\setminus\set{1'}},\tilde{c}')$, where $\tilde{c}$ is the restriction of $c$ to $W_J$ and $\tilde{c}'$ is the restriction of $c'$ to $W'_{\eta(J)\setminus\set{1'}}$.
\item $m'(\eta(r),\eta(s))\le m(r,s)$ for each pair $r,s\in S$ with $\eta(r)\neq1'$ and $\eta(s)\neq1'$.
\item \label{c' eta}
If $s_1s_2\cdots s_n$ is a reduced word for $c$, then $c'=\eta(s_1)\eta(s_2)\cdots\eta(s_n)$.
\end{enumerate}
\end{proposition}
\begin{proof}
The proof of Proposition~\ref{basic facts} can be repeated verbatim to prove all of the assertions except~\eqref{c' eta}.
The latter is equivalent to the statement that $\eta(r)$ precedes $\eta(s)$ if and only if $r$ precedes $s$, whenever $r,s\in S$ have $\eta(r)\neq1'$ and $\eta(s)\neq1'$.
This statement follows immediately from~\eqref{restrict camb} with $J=\set{r,s}$.
\end{proof}

Taking $\eta=\eta_J$ and $\Gamma=S\setminus J$ in Theorem~\ref{key camb}, we have the following analog of Theorem~\ref{para cong} for Cambrian lattices:

\begin{theorem}\label{camb para cong}
Let~$c$ be a Coxeter element of~$W,$ let $J\subseteq S$ and let~$\tilde c$ be the Coxeter element of~$W_J$ obtained by restriction.
Then $\eta_J$ restricts to a surjective lattice homomorphism from $\Camb(W,c)$ to $\Camb(W_J,\tilde c)$.
The associated lattice congruence on $\Camb(W,c)$ is generated by the set $S\setminus J$ of join-irreducible elements.
\end{theorem}

To prove Theorem~\ref{camb para factor}, the proof of Theorem~\ref{para factor} can be repeated verbatim, with Proposition~\ref{camb basic facts} replacing Proposition~\ref{basic facts} and Theorem~\ref{camb para cong} replacing Theorem~\ref{para cong}.
We prove Theorems~\ref{camb exist unique} and~\ref{camb diagram} together.

\begin{proof}[Proof of Theorems~\ref{camb exist unique} and~\ref{camb diagram}]
We first claim that there is a compressive homomorphism $\tilde\eta:W\to W'$ fixing $S$ and a generating set $\Gamma$ for the associated congruence such that $\tilde\Gamma=\Gamma\cap\Camb(W,c)$.
Arguing as in previous sections, Theorems~\ref{edge factor} and~\ref{edge cong ji} reduce the proof of the claim to the case where $W$ and $W'$ are irreducible and their diagrams coincide except for edge labels.
Looking through the type-by-type results of Sections~\ref{dihedral sec}--\ref{exceptional sec}, we see that in almost every case, the claim is true for a simple reason:  
For each choice of $\tilde\Gamma$, there is a surjective homomorphism from $W$ to $W'$ whose associated congruence is homogeneous of degree $2$, with a generating set that includes $\tilde\Gamma$.
The only exceptions come when $(W,W')$ is $(H_3,B_3)$ or $(H_4,B_4)$ and one of the following cases applies:

\noindent
\textbf{Case 1:} $q$ precedes $r$ and $\tilde\Gamma=\set{qr}$.
In this case, in light of Case 2 of the proof of Theorem~\ref{H3 B3 thm}, we need to verify that $qrqsrqrs\in H_3$ is not $qrs$-sortable and not $qsr$-sortable.
But $qrqsrqrs$ has only one other reduced word, $qrsqrqrs$, so its $qrs$-sorting word is $qrs|qr|qrs$, and thus it is not $qrs$-sortable.
Similarly, the $qsr$-sorting word for $qrqsrqrs$ is $qr|qsr|qr|s$, so it is not $qsr$-sortable.
The claim is proved in this case for $(W,W')=(H_3,B_3)$.
We easily conclude that $qrqsrqrs$ is not $qrst$-sortable, $qrts$-sortable, $qstr$-sortable or $qtsr$-sortable as an element of $H_4$ (cf. \cite[Lemma~2.3]{sort_camb}), so the claim is proved in this case for $(W,W')=(H_4,B_4)$ as well.

\noindent
\textbf{Case 2:} $r$ precedes $q$ and $\tilde\Gamma=\set{rq}$.
Arguing similarly to Case 1, it is enough to show that $qrqsrqrs$ is not $rqs$-sortable and not $srq$-sortable as an element of $H_3$.
This is true because its $rqs$-sorting word is $q|rqs|rq|rs$ and its $srq$-sorting word is $q|rq|srq|r|s$.

\noindent
\textbf{Case 3:} $q$ precedes $r$ and $\tilde\Gamma=\set{qrqr}$.
In light of Case 8 of the proof of Theorem~\ref{H3 B3 thm}, we need to verify that $rqsrqrs\in H_3$ is not $qrs$-sortable and not $qsr$-sortable.
This is true because its $qrs$-sorting word is $rs|qr|qrs$ and its $sqr$-sorting word is $r|sqr|qr|s$.

\noindent
\textbf{Case 4:} $r$ precedes $q$ and $\tilde\Gamma=\set{rqrq}$.
The $sqr$-sorting word for $rqsrqrs$ is $r|sqr|qr|s$, and the $srq$-sorting word for $rqsrqrs$ is $rq|srq|r|s$, so $rqsrqrs$ is neither $rqs$-sortable nor $srq$-sortable.
As in Case 3, this is enough.

This completes the proof of the claim.
Now Theorem~\ref{key camb} says that the restriction $\eta$ of $\tilde\eta$ is a compressive homomorphism from $\Camb(W,c)$ to $\Camb(W',c')$, whose associated congruence $\Theta$ is generated by $\tilde\Gamma$.
But then $\eta$ is the unique homomorphism from $\Camb(W,c)$ to $\Camb(W',c')$ whose congruence contracts $\tilde\Gamma$:
Any other congruence contracting $\tilde\Gamma$ is associated to a strictly coarser congruence, and thus the number of congruence classes is strictly less than $|\Camb(W',c')|$.
We have proved Theorem~\ref{camb exist unique}\eqref{exist unique tildeGamma}, Theorem~\ref{camb exist unique}\eqref{generated}, and Theorem~\ref{camb diagram}\eqref{exists diagram}.

Now let $\eta:\Camb(W,c)\to\Camb(W',c')$ be any surjective lattice homomorphism.
Theorem~\ref{camb basic facts}\eqref{restrict camb} and Theorem~\ref{camb basic facts}\eqref{c' eta} imply that, for each $r,s\in S$ with $m'(r,s)<m(r,s)$ such that $r$ precedes $s$ in $c$, the congruence associated to $\eta$ contracts exactly $m(r,s)-m'(r,s)$ join-irreducible elements of the form $\alt_k(r,s)$ with $k=2,\ldots,m(r,s)-1$.
Theorem~\ref{camb exist unique}\eqref{every arises} follows by the uniqueness in Theorem~\ref{camb exist unique}\eqref{exist unique tildeGamma}.

Finally, we prove Theorem~\ref{camb diagram}\eqref{any diagram}.
Given any $\tilde\eta$, let $\Gamma$ be a minimal generating set for the associated congruence.
Theorem~\ref{key camb} says that the restriction of $\eta$ to $\Camb(W,c)$ is a surjective homomorphism to $\Camb(W',c')$, and the associated congruence is generated by the $\Gamma\cap\Camb(W,c)$.
But by the uniqueness in Theorem~\ref{camb exist unique}\eqref{exist unique tildeGamma} and by Theorem~\ref{camb exist unique}\eqref{generated}, the restriction of $\tilde\eta$ to $\Camb(W,c)$ is~$\eta$.
\end{proof}

\section{Refinement relations among Cambrian fans}\label{crys case}
Given a finite crystallographic Coxeter group $W$ with Coxeter arrangement $\A$, an associated Cartan matrix $A$, and a Coxeter element $c$ of $W$, the \newword{Cambrian fan} for $(A,c)$ is the fan defined by the shards of $\A$ not removed by the Cambrian congruence $\Theta_c$.
In this section, we prove Theorem~\ref{camb fan coarsen}, which gives explicit refinement relations among Cambrian fans.
We assume the most basic background about Cartan matrices of finite type and the associated finite root systems.

Recall from the introduction that a Cartan matrix $A=[a_{ij}]$ \newword{dominates} a Cartan matrix $\A'=[a'_{ij}]$ if $|a_{ij}|\ge |a'_{ij}|$ for all $i$ and~$j$.
Recall also that Proposition~\ref{dom subroot} says that when $A$ dominates $A'$, $\Phi(A)\supseteq\Phi(A')$ and $\Phi_+(A)\supseteq\Phi_+(A')$, assuming that $\Phi(A)$ and $\Phi(A')$ are both defined with respect to the same simple roots $\alpha_i$.
We again emphasize that $\Phi(A)$ includes any imaginary roots.
The proposition appears as \cite[Lemma~3.5]{Marquis}, but for completeness, we give a proof here.

\begin{proof}[Proof of Proposition~\ref{dom subroot}]
To construct a Kac-Moody Lie algebra from a symmetrizable Cartan matrix $A$, as explained in \cite[Chapter~1]{Kac}, one first defines an auxiliary Lie algebra $\tfg(A)$ using generators and relations.
The Lie algebra $\tfg(A)$ decomposes as a direct sum $\tfn_-\oplus\fh\oplus\tfn_+$ \emph{of complex vector spaces}.
We won't need details on $\fh$ here, but its dual contains linearly independent vectors $\alpha_1,\ldots,\alpha_n$ called the \newword{simple roots}.
The summand $\tfn_+$ decomposes further as a (\emph{vector space}) direct sum with infinitely many summands $\tfg_\alpha$, indexed by nonzero vectors $\alpha$ in the nonnegative integer span $Q_+$ of the simple roots.
Similarly, $\tfn_-$ decomposes into summands indexed by nonzero vectors in the nonpositive integer span of the simple roots.
Furthermore, $\tfn_+$ is freely generated by elements $e_1,\ldots,e_n$ and thus is independent of the choice of $A$ (as long as $A$ is $n\times n$).
Similarly, $\tfn_-$ is freely generated by elements $f_1,\ldots,f_n$ and is independent of the choice of $A$.

There is a unique largest ideal $\fr$ of $\tfg(A)$ whose intersection with $\fh$ is trivial, and this is a direct sum $\fr_-\oplus\fr_+$ \emph{of ideals} with $\fr_\pm=\fr\cap\fn_\pm$.
The \newword{Kac-Moody Lie algebra} $\fg(A)$ is defined to by $\tfg(A)/\fr$.
The Lie algebra $\fg(A)$ inherits a direct sum decomposition $\fn_-\oplus\fh\oplus\fn_+$, and the summands $\fn_\pm$ decompose further as 
\[\fn_-=\bigoplus_{0\neq\alpha\in Q_+}\fg_{-\alpha}\qquad\text{and}\qquad \fn_+=\bigoplus_{0\neq\alpha\in Q_+}\fg_\alpha.\]
A \newword{root} is a nonzero vector $\alpha$ in $Q_+\cup(-Q_+)$ such that $\fg_\alpha\neq0$.
The \newword{(Kac-Moody) root system} associated to $A$ is the set of all roots.

While $\tfn_-$ and $\tfn_+$ are independent of the choice of $A$, the ideal $\fr$ depends on $A$ (in a way that is not apparent here because we have not given the presentation of $\tfg(A)$).
Thus, writing $\fr(A)$ for the ideal associated to $A$, Proposition~\ref{dom subroot} follows immediately from this claim:  
If $A$ and $A'$ are Cartan matrices such that $A$ dominates~$A'$, then $\fr(A)\subseteq\fr(A')$.

This claim follows immediately from a description of the ideal $\fr$ in terms of the \newword{Serre relations}.
Recall that for $x$ in a Lie algebra $\fg$, the linear map $\ad_x:\fg\to\fg$ is $y\mapsto[x,y]$.
In \cite[Theorem~9.11]{Kac}, it is proved that the ideal $\fr_+$ is generated by $\set{(\ad_{e_i})^{1-a_{ij}}(e_j):i\neq j}$ and that $\fr_+$ is generated by $\set{(\ad_{f_i})^{1-a_{ij}}(f_j):i\neq j}$.
\end{proof}

Although we have proved Proposition~\ref{dom subroot} in full generality, we pause to record the following proof due to Hugh Thomas (personal communication, 2018) in the case where $A$ is symmetric.

\begin{proof}[Quiver-theoretic proof of Proposition~\ref{dom subroot} (for $A$ symmetric)]
Suppose that $Q$ is a quiver with $n$ vertices, having no loops (i.e.\ $1$-cycles), suppose $v\in\integers^n$, and fix an algebraically closed field.
We associate a symmetric Cartan matrix $A$ to $Q$ by ignoring the direction of edges and taking $a_{ij}$ to be the total number of edges connecting vertex $i$ to vertex $j$.

Kac \cite[Theorem~3]{Kac1980} (generalizing Gabriel's Theorem) proved that $v$ is the dimension vector of an indecomposable representation of $Q$ if and only if $v$ is the simple-root coordinate of a positive root in the root system $\Phi(A)$.

Given symmetric Cartan matrices $A$ and $A'$ with $A$ dominating $A'$, we can construct a quiver $Q'$ associated to $A'$ and add arrows to it to obtain a quiver $Q$ associated to $A$.
If $v$ is the simple root coordinates of a root in $\Phi(A')$, then if we start with an indecomposable representation of $Q'$ with dimension vector $v$ and assign arbitrary maps to the arrows in $Q$ that are not in $Q'$, the result is still indecomposable, so $v$ is the simple root coordinates of a root in $\Phi(A)$.
\end{proof}

The dual root system $\Phi\ck(A)$ consists of all co-roots associated to the roots in $\Phi(A)$.
This is a root system in its own right, associated to $A^T$.
Since $A$ dominates $A'$ if and only if $A^T$ dominates $(A')^T$, the following proposition is immediate from Proposition~\ref{dom subroot}.

\begin{proposition}\label{dom subcoroot}
Suppose $A$ and $A'$ are symmetrizable Cartan matrices such that $A$ dominates $A'$.
If $\Phi\ck(A)$ and $\Phi\ck(A')$ are both defined with respect to the same simple co-roots $\alpha\ck_i$, then $\Phi\ck(A)\supseteq\Phi\ck(A')$ and $\Phi\ck_+(A)\supseteq\Phi\ck_+(A')$.
\end{proposition}

Suppose $A$ dominates $A'$ and suppose $(W,S)$ and $(W',S)$ are the associated Coxeter systems.
Let $\A$ and $\A'$ be the associated Coxeter arrangements, realized so that the simple \emph{co}-roots are the same for the two arrangements.
(This requires that two different Euclidean inner products be imposed on $\reals^n$.)
Proposition~\ref{dom subcoroot} implies that $\A'\subseteq\A$, so that in particular, each region of $\A$ is contained in some region of $\A'$.
Since the regions of $\A$ are in bijection with the elements of $W$ and the regions of $\A'$ are in bijection with the elements of $W'$, this containment relation defines a surjective map $\eta:W\to W'$.

\begin{theorem}\label{weak hom subroot}
The map $\eta$, defined above, is a surjective lattice homomorphism from the weak order on $W$ to the weak order on $W'$.
\end{theorem}
\begin{proof}
As in the proof of Proposition~\ref{dom subroot}, we may restrict our attention to the case where $A$ is irreducible and barely dominates $A'$.
We first consider the case where $A'$ is obtained by erasing a single edge $e$ in the Dynkin diagram of $A$.
Let $E=\set{e}$.
Recall from Section~\ref{erase edge sec} that the map $\eta_E$ maps $w\in W$ to $(w_I,w_J)\in W_I\times W_J=W'$, where $I$ and $J$ are the vertex sets of the two components of the diagram of $A'$.
The set of reflections in $W'$ equals the set of reflections in $W_I$ union the set of reflections in $W_J$.
Since $\inv(w_I)=\inv(w)\cap W_I$ and $\inv(w_J)=\inv(w)\cap W_J$, two elements of $W$ map to the same element of $W'$ if and only if the symmetric difference of their inversion sets does not intersect $W_I\cup W_J$.
Comparing with the edge-erasing case in the proof of Proposition~\ref{dom subroot}, and noting that both maps fix $W'$, we see that $\eta=\eta_E$.
Theorem~\ref{edge cong} now says that $\eta$ is a surjective homomorphism.

The case where $A$ is of type $G_2$ is easy and we omit the details.

Now suppose $A$ is of type $C_n$, so that the dual root system $\Phi\ck(A)$ is of type $B_n$.
Using a standard realization of the type-B root system, reflections in $W$ correspond to positive co-roots as follows:
A co-root $e_i$ corresponds to $(-i\,\,\,i)$, a co-root $e_j-e_i$ corresponds to $(-j\,\,\,-i)(i\,\,\,j)$, and a co-root $e_j+e_i$ corresponds to $(i\,\,\,-j)(j\,\,\,-i)$.
The positive co-roots for $A'$ are the roots $\alpha_1+\alpha_2+\cdots+\alpha_j=e_j$ for $1\le j\le n$ and $\alpha_i+\alpha_{i+1}+\cdots+\alpha_j=e_j-e_{i-1}$ for $2\le i\le j\le n$.
Two adjacent $\A$-regions are contained in the same $\A'$-region if and only if the hyperplane separating them is in $\A\setminus\A'$.
Thus, in light of Proposition~\ref{simion cover}, we see that $\eta$ has the same fibers as the map $\eta_\sigma$ of Section~\ref{simion sec}, which is a surjective lattice homomorphism by Theorem~\ref{simion thm}.
Writing, as before, $\tau$ for the inverse map $\tau$ to $\eta_\sigma$, the inversions of a permutation $\pi$ correspond to the inversions $t$ of $\tau(\pi)$ such that $H_t\in\A'$.
It follows that $\eta=\eta_\sigma$.

When $A$ is of type $B_n$, so that $\Phi\ck(A)$ is of type $C_n$, the argument is the same, using Proposition~\ref{nonhom cover} instead of Proposition~\ref{simion cover}.

When $A$ is of type $F_4$, the dual root system $\Phi\ck(A)$ is also of type $F_4$, and we choose an explicit realization of $\Phi\ck(A)$ with
\[\alpha\ck_p=\frac12(-e_1-e_2-e_3+e_4)\,,\quad\alpha\ck_q=e_1\,,\quad\alpha\ck_r=e_2-e_1\,,\quad\alpha\ck_s=e_3-e_2.\]
Here $p$, $q$, $r$, and $s$ are as defined in connection with Theorem~\ref{F4 thm}.
The positive co-roots are $\set{\frac12(\pm e_1\pm e_2\pm e_3+e_4)}\cup\set{e_i:i=1,2,3,4}\cup\set{e_j\pm e_i:1\le i<j\le 4}$.
The positive co-roots for the $A'$, as a subset of these co-roots, are 
\[\set{e_1,e_2,e_3}\cup\set{e_2-e_1,e_3-e_2,e_3-e_1}\cup\SEt{\sum_{i=1}^4b_ie_i:b_i\in\Set{\pm\frac12},b_4=\frac12\sum b_i\le0}\]
Computer calculations show that the surjective homomorphism $\eta$ found in Case~4 of the proof of Theorem~\ref{F4 thm} has the property that two elements of $W$, related by a cover in the weak order, map to the same element of $W'$ if and only if the reflection that relates them corresponds to a co-root not contained in the subset $\Phi\ck(A')$.
The two maps coincide by the same argument given in previous cases.
\end{proof}

Theorems~\ref{camb exist unique} and~\ref{weak hom subroot} immediately imply Theorem~\ref{camb fan coarsen}.

\begin{remark}\label{alt proof}
We mention two alternative proofs of Theorem~\ref{weak hom subroot} that are almost uniform, but not quite.
Both use the following fact:
For $A$ and $A'$ as in Proposition~\ref{dom subroot}, if $\Phi(A)$ is finite, then $\Phi(A')$ is not only a subset of $\Phi(A)$, but also an order ideal in the root poset of $\Phi(A)$ (the positive roots ordered with $\alpha\le\beta$ if and only if $\beta-\alpha$ is in the nonnegative span of the simple roots).
This fact is easily proved type-by-type, but we are unaware of a uniform proof.
Indeed, the fact fails for infinite root systems, so perhaps a uniform proof shouldn't be expected.
Using the fact, it is easy to prove Theorem~\ref{weak hom subroot} using the polygonality of the weak order (\cite[Theorem~10-3.7]{regions10} and \cite[Theorem~9-6.5]{regions9}) or using the characterization of inversion sets as rank-two biconvex (AKA biclosed) sets of positive roots (\cite[Lemma~4.1]{Dyer} or \cite[Theorem~10-3.24]{regions10}).
\end{remark}

\section*{Acknowledgments}
Thanks are due to Vic Reiner, Hugh Thomas, and John Stembridge for helpful conversations.
Thanks in particular to Hugh Thomas for providing the quiver-theoretic proof of the symmetric case of Proposition~\ref{dom subroot} (see Section~\ref{crys case}) and encouraging the author to improve upon an earlier non-uniform finite-type proof.
Thanks to Nicolas Perrin for bibliographic help concerning Kac-Moody Lie algebras and for his helpful notes~\cite{Perrin}. 
The computations described in this paper were carried out in \texttt{maple} using Stembridge's \texttt{coxeter/weyl} and \texttt{posets} packages~\cite{StembridgePackages}.
An anonymous referee gave many helpful suggestions and pointed out most of what appears in Remark~\ref{alt proof}.

\end{document}